\documentclass[10pt]{amsart}
\usepackage{color}
\usepackage{enumerate}
\usepackage{amssymb}
\usepackage{multirow}
\usepackage{tikz}
\usetikzlibrary{shapes.geometric, arrows}

\tikzstyle{startstop}=[rectangle, rounded corners, minimum width=3cm, minimum height=1cm,text centered, draw=black]
\tikzstyle{io}=[trapezium, trapezium left angle=70, trapezium right angle=110, minimum width=3cm, minimum height=1cm, text centered, draw=black, fill=blue!30]
\tikzstyle{process}=[rectangle, minimum width=3cm, minimum height=1cm, text centered, draw=black, fill=orange!30]
\tikzstyle{decision} = [diamond, minimum width=3cm, minimum height=1cm, text centered, draw=black, fill=green!30]
\tikzstyle{arrow} = [thick,->,>=stealth]

\makeatletter 
\def\@tocline#1#2#3#4#5#6#7{\relax
  \ifnum #1>\c@tocdepth 
  \else
    \par \addpenalty\@secpenalty\addvspace{#2}%
    \begingroup \hyphenpenalty\@M
    \@ifempty{#4}{%
      \@tempdima\csname r@tocindent\number#1\endcsname\relax
    }{%
      \@tempdima#4\relax
    }%
    \parindent\z@ \leftskip#3\relax \advance\leftskip\@tempdima\relax
    \rightskip\@pnumwidth plus4em \parfillskip-\@pnumwidth
    #5\leavevmode\hskip-\@tempdima
      \ifcase #1
       \or\or \hskip 1em \or \hskip 2em \else \hskip 3em \fi%
      #6\nobreak\relax
    \dotfill\hbox to\@pnumwidth{\@tocpagenum{#7}}\par
    \nobreak
    \endgroup
  \fi}
\makeatother

\newtheorem{theorem}{Theorem}[section]

\newtheorem{lemma}[theorem]{Lemma}
\newtheorem{proposition}[theorem]{Proposition}
\newtheorem{corollary}[theorem]{Corollary}

\theoremstyle{definition}
\newtheorem{definition}[theorem]{Definition}
\newtheorem{example}[theorem]{Example}
\newtheorem{remark}[theorem]{Remark}

\newcommand{\de}{\partial}
\newcommand{\R}{\mathbb R}

\newcommand{\C}{\mathbb C}

\newcommand{\So}{\mathcal{SR}}

\newcommand{\SF}{\mathbb S}

\newcommand{\HH}{\mathbb H}

\def\Re{{\sf Re}}
\def\Im{{\sf Im}\,}

\numberwithin{equation}{section}

\begin{document}

\title[Spherical Coefficients]{Spherical coefficients of slice regular functions}
\author[A. Altavilla]{Amedeo Altavilla}\address{Altavilla Amedeo:  Dipartimento di Matematica, Universit\`a degli Studi di Bari ``Aldo Moro'', via Edoardo Orabona, 4, 70125,
Bari, Italy.}\email{amedeo.altavilla@uniba.it}

\thanks{Partially supported by GNSAGA of INdAM and by the INdAM project ``Teoria delle funzioni ipercomplesse e applicazioni''. I warmly thank the anonymous referee whose precious comments and suggestions helped improve the presentation of this paper.}



\begin{abstract} 
Given a quaternionic slice regular function $f$, we give a direct and effective way to compute the coefficients
of its spherical expansion at any point. Such coefficients are obtained in terms of spherical and slice derivatives of the function itself. 
Afterwards, we compare the coefficients of $f$ with those of its slice derivative $\de_{c}f$ obtaining a countable family of
differential equations satisfied by any slice regular function. The results are proved in all details and 
are accompanied to several examples.

For some of the results, we also give alternative proofs.
\end{abstract}
\maketitle
\tableofcontents

\section{Introduction}
In basic complex analysis, if $D\subset\C$ is an open domain and $f:D\to\C$ is a holomorphic function, then for any $z_{0}\in D$ there exists a neighborhood $U\subset D$ of $z_{0}$
where $f$ can be written in the following form
$$
f(z)=\sum_{n\geq 0}(z-z_{0})^{n}c_{n}.
$$
A natural way to obtain the coefficient $c_{n}\in\C$ is to compute the $n$-th complex derivative $f$ and evaluate it at $z_{0}$, i.e. $c_{n}=\frac{1}{n!}(\de_{z}^{n}f)(z_{0})$.
In particular, if $f$ is a polynomial of degree $m$, then $c_{N}=0$ for all $N>m$. Moreover by comparing the coefficients $c_{n}$ of $f$ with those $c_{n}'$
of $\de_{z}f$ we get that $c_{n}'=(n+1)c_{n+1}$, i.e. the following silly equation $\de_{z}^{k}(\de_{z}f)=\de_{z}^{k+1}f$. In this paper we will look for
proper generalizations of the previous basic features in the setting of quaternionic slice regularity.

The concept of slice regular function was introduced in~\cite{G-St} to build a theory of quaternionic regular functions that resemble somehow complex holomorphy
and to contain quaternionic polynomials with right (or left) coefficients. 
Quite immediately a Taylor-like expansion centered at real points was obtained and later, with
suitable modifications on the domain of convergence, also at non-real points (see~\cite[Chapter 2]{G-S-St} and references therein).
At a later time, C.~Stoppato~\cite{stoppatosph} showed the possibility to expand any slice regular function on an Euclidean open subset of its domain near any 
point, using suitable polynomials (see also~\cite[Chapter 8.1]{G-S-St}). 
What we will perform in this paper is to show a new effective way to compute the coefficients of such expansion. We will show several examples and obtain, as a byproduct, a 
countable family of differential equations satisfied by slice regular functions.

We start with some standard material setting down the notations, the main definitions and results.
The algebra of quaternions will be denoted by $\HH$. Any element $x\in\HH$ is of the form $x=x_{0}+x_{1}i+x_{2}j+x_{3}k$, where $i,j$ and $k$, follow the usual quaternionic
multiplicative rules. For any $x\in\HH$ we set $x^{c}=x_{0}-(x_{1}i+x_{2}j+x_{3}k)$ to be its \textit{conjugate} and $\Re(x)=(x+x^{c})/2$, $\Im(x)=(x-x^{c})/2$ to be its
real and imaginary part, respectively. The norm of any quaternion $x$ is defined to be $|x|=\sqrt{x\cdot x^{c}}\in\R$ and, if $x\neq 0$ its inverse is given by 
$x^{-1}=x^{c}/|x|^{2}$.
Now, $i,j$ and $k$ are not the only imaginary units in $\HH$, indeed, for any $\alpha_{1},\alpha_{2},\alpha_{3}\in\R$ such that $\alpha_{1}^{2}+\alpha_{2}^{2}+\alpha_{3}^{2}=1$,
we have that $\alpha_{1}i+\alpha_{2}j+\alpha_{3}k$ is an imaginary unit as well. For this reason we introduce the following notation
$$
\SF:=\{x\in\HH\,|\,x^{2}=-1\}=\{\alpha_{1}i+\alpha_{2}j+\alpha_{3}k\,|\,\alpha_{1},\alpha_{2},\alpha_{3}\in\R,\,\alpha_{1}^{2}+\alpha_{2}^{2}+\alpha_{3}^{2}=1\}.
$$
Thanks to the previous observation, it is easy to see that 
$$
\HH=\bigcup_{I\in\SF}\C_{I},
$$
where $\C_{I}=\{\alpha+\beta I\in\HH\,|\,\alpha,\beta\in\R\}$, showing the \textit{slice} nature of $\HH$.
In particular, any $x\in\HH$ can be written as
$$
x=\Re(x)+\Im(x)=\Re(x)+\frac{\Im(x)}{|\Im(x)|}|\Im(x)|=\alpha+I\beta,
$$
where $\alpha=\Re(x)$, $\beta=|\Im(x)|$, $I=\Im(x)/|\Im(x)|$ and, if $\Im(x)=0$, then of course $\beta=0$ and $I$ is any element in $\SF$.

We now pass to define the class of slice regular functions following the approach of~\cite{G-P}.
Let $D\subset\C$ be an open domain such that $D=\bar D$ and $\Omega_{D}$ its circularization defined as
$$
\Omega_{D}:=\{\alpha+I\beta\in\HH\,|\,\alpha+i\beta\in D,\,I\in\SF\}.
$$
Sets of the form $\Omega_{D}$ will be called \textit{symmetric sets}.
The definition of slice regularity needs one last preliminary observation: the real tensor product $\HH\otimes\C$ inherits from $\C$ a natural complex structure, namely if
$$\HH\otimes\C=\{x+\sqrt{-1}y\,|\,x,y\in\HH\},$$
then such a complex structure is exactly the multiplication by $\sqrt{-1}$. Analogously, the standard conjugation in $\C$ induces the following conjugation in $\HH\otimes\C$:
we have $\overline{x+\sqrt{-1}y}=x-\sqrt{-1}y$. 

\begin{definition}
Let $D\subset\C$ be an open domain such that $D=\bar D$ and $\Omega_{D}\subset\HH$ denotes its circularization. 
\begin{enumerate}
\item A function $F=F_{0}+\sqrt{-1}F_{1}: D\to\HH\otimes\C$ is said to be a \textit{stem function} if $F(\bar z)=\overline{F(z)}$, for any $z\in D$.
\item Any stem function $F=F_{0}+\sqrt{-1}F_{1}:D\to\HH\otimes\C$ induces a (unique) \textit{left slice function} $f:\Omega_{D}\to\HH$ in the following way
$$
f(\alpha+I\beta)=F_{0}(\alpha+i\beta)+IF_{1}(\alpha+I\beta).
$$
We will write $f=\mathcal{I}(F)$ to indicate that $F$ induces $f$. The set of slice functions defined on $\Omega_{D}$ of class $\mathcal{C}^{k}$, with $k\in\mathbb{N}\cup\{\omega\}$, will be denoted
by $\mathcal{S}^{k}(\Omega_{D})$.
\item Let $f=\mathcal{I}(F)\in\mathcal{S}^{1}(\Omega_{D})$ be any slice function of class $\mathcal{C}^{1}$. We say that $f$ is \textit{left slice regular} if its defining stem function
is holomorphic with respect to the standard complex structure in $\C$ and the one previously defined in $\HH\otimes\C$. The set of slice regular functions defined on $\Omega_{D}$
will be denoted by $\So(\Omega_{D})$.
\end{enumerate}
\end{definition}
In what follows we will omit the adjective ``left'' when talking of sliceness or slice regularity. The word ``unique'' in the parenthesis of point $(2)$ is justified by the fact that the map $F\mapsto\mathcal{I}(F)$ is bijective~\cite{G-P}. An analogous theory of \textit{right} slice regular function can be obviously defined.
This definition of regularity is equivalent to the one given in~\cite{G-St, G-S-St} on domains $\Omega=\Omega_{D}$, with
$D\cap\mathbb{R}\neq \emptyset$. New types of domains are considered in~\cite{douren1,douren2,dourensabadini, G-SJMAA,gentilistoppatoPAMS}.  The theory has also been generalized to the several variable case~\cite{ghiloniperottiseveral1,ghiloniperottiseveral2}. Very recently the definition of regularity was further 
analyzed in~\cite{mongodi} and in~\cite{ghilonipoly}.

We now introduce differential operators related to slice regularity.
\begin{definition}\label{sphericalderivativevalue}
Let $D\subset\C$ be an open domain such that $D=\bar D$ and $\Omega_{D}\subset\HH$ denotes its circularization, $x=\alpha+I\beta\in\Omega_{D}$ and $z\in D$ be such that $z=\alpha+i\beta$. 
\begin{enumerate}
\item  Let $f=\mathcal{I}(F_{0}+\sqrt{-1}F_{1}):\Omega_{D}\to\HH$ be any slice function. We define the \textit{spherical value} of $f$ as the slice function $v_{s}f:\Omega_{D}\to\HH$ defined by
$$
(v_{s}f)(x)=\mathcal{I}(F_{0})(x)=\frac{f(x)+f(x^{c})}{2}=F_{0}(\alpha+i\beta).
$$
We define the \textit{spherical derivative} of $f$ as the slice function $\de_{s}f:\Omega_{D}\setminus\R\to\HH$ defined by
$$
(\de_{s}f)(x)=\mathcal{I}\left(\frac{2iF_{1}}{z-\bar z}\right)(x)=(2\Im(x))^{-1}(f(x)-f(x^{c}))=\frac{F_{1}(\alpha+i\beta)}{\beta}.
$$
\item Let $f=\mathcal{I}(F)\in\mathcal{S}^{1}(\Omega_{D})$ be any slice function of class $\mathcal{C}^{1}$. We define the functions $\bar\de_{c}f,\de_{c}f:\Omega_{D}\to\HH$ as
$$
(\bar\de_{c}f)(x)=\mathcal{I}\left(\frac{\de F}{\de \bar z}\right)(x),\qquad (\de_{c}f)(x)=\mathcal{I}\left(\frac{\de F}{\de z}\right)(x).
$$
The function $\partial_{c}f$ is called the \textit{slice derivative} of $f$.
\end{enumerate}
\end{definition}

It was proven in~\cite{Perotti} that the spherical derivative operator is indeed a differential operator when applied to slice functions. It is in fact equal to some other differential operator
related to the operators $\bar\partial_{c}, \partial_{c}$ and to the Cauchy-Dirac operator. Clearly, a differentiable slice function $f$ is slice regular if and only if $\bar\de_{c}f$ vanishes identically.

Regarding the notation, in recent works the spherical value, the spherical derivative and the slice derivatives are often written as
$$
v_{s}f=f_{s}^{\circ},\qquad\de_{s}f=f_{s}',\qquad\bar\de_{c}f=\frac{\de f}{\de x^{c}},\qquad\de_{c}f=\frac{\de f}{\de x}.
$$
We will use the older notations because they fit better in the computations we are going to perform.

It is well known that the notion of slice regularity is not closed under pointwise product: it is enough to consider the functions $f(x)=xa$ and $g(x)=x$ for a suitable $a\in\HH$; 
the point wise product $(fg)(x)=f(x)g(x)=xax$ it is not even a slice function in general. However, a natural modification of the pointwise product that preserves regularity can be given.
\begin{definition}
Let $f=\mathcal{I}(F)$ and $g=\mathcal{I}(G)$ be two slice functions defined on the same domain $\Omega_{D}$. We define the \textit{slice product} of $f$ and $g$ as
the slice function $f\cdot g:\Omega_{D}\to \HH$ defined as
$$
(f\cdot g)(x)=\mathcal{I}(FG)(x).$$
\end{definition}
We have that if $f$ and $g$ are slice regular functions then their slice product is regular too.
Some classes of slice functions play a remarkable role in relation with the slice product. These are the so called slice preserving and one-slice preserving functions.
\begin{definition} Let $f=\mathcal{I}(F_{0}+\sqrt{-1}F_{1}):\Omega_{D}\to\HH$ be any slice function. We say that $f$ is a \textit{slice preserving} function if $F_{0}$ and $F_{1}$ are real valued functions,
while, if $F_{0}$ and $F_{1}$ take their values in some fixed $\C_{J}\subset\HH$, for some $J\in\SF$, then $f$ is said to be a \textit{one-slice preserving} function or a $\C_{J}$-\textit{preserving} function.
\end{definition}
Let $f:\Omega_{D}\to\HH$ be any slice function. Then, clearly if $f$ is $\C_{J}$-preserving, for some $J$, then
$f(\Omega_{D}\cap\C_{J})\subset \C_{J}$, while if $f$ is slice preserving, then $f(\Omega\cap\C_{I})\subset\C_{I}$, for any $I\in\SF$.
\begin{remark}
Let $f=\mathcal{I}(F)$ and $g=\mathcal{I}(G)$ be two slice functions defined on the same domain $\Omega_{D}$. If $f$ is slice preserving we have that $f\cdot g=g\cdot f=fg$.
If $f$ and $g$ are both $\C_{J}$-preserving, for some $J$, then $f\cdot g=g\cdot f$ and $(f\cdot g)_{|\Omega_{D}\cap\C_{J}}=fg$.
\end{remark}
Algebraic properties of such classes of functions are widely studied in~\cite{A-dF,A-dFAMPA,A-dFLAA,A-dFConc}.

We now turn back to the main object of this paper. At the beginning of this introduction we talked about a Taylor-like expansion obtained for slice regular functions.
This was obtained by means of the use of the polynomials $(x-q_{0})$ and its slice powers $(x-q_{0})^{\cdot n}$. The other approach exploits instead the following polynomials.
\begin{definition}\label{characteristicpolynomial}
Let $q_{0}=\alpha+I\beta$ be any non-real quaternion. We define the \textit{characteristic polynomial} of $q_{0}$ as the slice regular polynomial $\Delta_{q_{0}}:\HH\to\HH$ defined
by 
$$
\Delta_{q_{0}}(x)=(x-q_{0})\cdot (x-q_{0}^{c})=x^{2}-x(q_{0}+q_{0}^{c})+q_{0}q_{0}^{c}=x^{2}-2\Re(q_{0})x+|q_{0}|^{2}=x^{2}-2x\alpha+\alpha^{2}+\beta^{2}.
$$
\end{definition}

For any $q_{0}=\alpha+I\beta\in\HH\setminus\R$ we have that $\Delta_{q_{0}}$ is a slice preserving function which identically vanishes on the $2$-sphere $\SF_{q_{0}}=\alpha+\SF \beta$.

We recall the following result.
\begin{theorem}\cite[Theorem 4.1]{stoppatosph}\label{stoppatothm}
Let $f$ be a quaternionic slice regular function on a symmetric domain $\Omega_{D}$ such that $\Omega_{D}\cap\R\neq\emptyset$ and let $q_{0}=\alpha_{0}+I_{0}\beta_{0}\in\Omega$ and $R>0$ be such that
$$
U(q_{0},R):=\{x\in\HH\,|\,|\Delta_{q_{0}}(x)|<R^{2}\}\subset \Omega_{D}.
$$
Then there exist $s_{n}\in\HH$ such that
\begin{equation}\label{sphericalexpansion}
f(x)=\sum_{n\in\mathbb{N}}\Delta_{q_{0}}^{n}(x)[s_{2n}+(x-q_{0})s_{2n+1}],
\end{equation}
for all $x\in U(q_{0},R)$.
\end{theorem}
Sets of type $U(q_{0},R)$ are called \textit{Cassini balls}.
\begin{definition}
Let $f:\Omega_{D}\to\HH$ be any slice regular function and $q_{0}\in\Omega_{D}\setminus \R$. The expression on the right hand side of Formula~\eqref{sphericalexpansion}
is called \textit{spherical expansion} of $f$ at $q_{0}$ and its coefficients $s_{n}\in\HH$ are called \textit{spherical coefficients} of $f$ at $q_{0}$.
\end{definition}

Such an expansion was later obtained also in other, more general, contexts~\cite{G-SJMAA,G-P-series,GPSadvances,GPSdivision} allowing to drop
the hypothesis $\Omega_{D}\cap\R\neq\emptyset$.
However, in none of these work
an immediate and clear representation of the spherical coefficients in terms of some differential operator applied to $f$ was given. We shortly review some results.
At first~\cite{stoppatosph,G-S-St} $s_{n}$ were given in terms of suitable quotients $(R_{\bar q_{0}} R_{q_{0}})^{n}f(q_{0})$ and $R_{q_{0}}(R_{\bar q_{0}} R_{q_{0}})^{n}f(\bar q_{0})$ defined as follows
\begin{align*}
f(x)&=f(q_{0})+(x-q_{0})\cdot R_{q_{0}}f(x)\\
f(x)&=f(q_{0})+(x-q_{0})\cdot R_{q_{0}}f(\bar q_{0})+[(x-\alpha_{0})^{2}+\beta_{0}^{2}]R_{\bar q_{0}}R_{q_{0}}f(x)\\
f(x)&=f(q_{0})+(x-q_{0})\cdot R_{q_{0}}f(\bar q_{0})+[(x-\alpha_{0})^{2}+\beta_{0}^{2}]R_{\bar q_{0}}R_{q_{0}}f(q_{0})+[(x-\alpha_{0})^{2}+\beta_{0}^{2}](x-q_{0})R_{q_{0}}R_{\bar q_{0}}R_{q_{0}}f(x)\\
\dots&
\end{align*}
This representation was sufficient to obtain Theorem~\ref{stoppatothm}, but is quite hard to excerpt some geometryc or analytic information from it.

Later, a quite standard integral representation~\cite[Theorem 5.4]{G-P-series} together with a 
more sophisticated~\cite[Theorem 3.7]{G-P-series} one involving slice derivatives  and an infinite matrix of coefficients was given.
In particular the authors write the $n$-th slice derivative of the function $f$ evaluated at $q_{0}$ as a suitable linear combination of a finite number of spherical coefficients of $f$ at $q_{0}$.

Another relation between spherical coefficients and slice derivatives was recently given in~\cite[Theorem 8.1]{GPSadvances}, but still not handy enough to actually compute coefficients.

In what follows we will give a new way to obtain the spherical coefficients of a slice regular function $f$ defined on a symmetric domain $\Omega_{D}$ that is somehow more convenient in explicit computations. This method simply involves spherical and slice derivatives of a slice regular functions, computed alternately. We will
be more clear in a moment. Firstly we describe the structure of the paper. 
The next section contains mostly computational preliminaries on the action of $\de_{s}$, $\de_{c}$, $(\de_{c}\de_{s})^{k}$ on a slice (regular) function. We can say that the main result in this section is a
general formula to compute $(\de_{c}\de_{s})^{k}$ of a slice product (see Proposition~\ref{main}). Thanks to this formula and to some other observation we are able
to determine in Lemma~\ref{lemmapowers} and Corollary~\ref{cor1}
how many alternated actions of $\partial_{c}$ and $\partial_{s}$ are needed to kill specific slice regular polynomials.

In Section~\ref{alternate}, we apply the result obtained in Proposition~\ref{main} to obtain the exact values of $((\de_{c}\de_{s})^{k}\Delta_{q_{0}}^{p}(x))$, 
$(\de_{s}(\de_{c}\de_{s})^{k}\Delta_{q_{0}}^{p}(x))$, $((\de_{c}\de_{s})^{k}\Delta_{q_{0}}^{p}(x)(x-q_{0}))$ and of $(\de_{s}(\de_{c}\de_{s})^{k}\Delta_{q_{0}}^{p}(x)(x-q_{0}))$
at $x=q_{0}$, for $q_{0}\in\HH\setminus\R$ and any $k,p\in\mathbb{N}$. At the end we add a convenient table containing all the non-vanishing results and the respective references 
(see Subsection~\ref{tabel1}).
We remark that, in the presentation of this section, not only the correctedness of the of the results and of their proofs was taken into account, but also their sequence.
Indeed, after Proposition~\ref{main}, the order in which the results are presented can be modified only by 
small changes (see the flow chart in Appendix~\ref{flow}).

Afterwards, in Section~\ref{coefficients} we state our first main result Theorem~\ref{sphcoeff}: we prove that the spherical coefficients of any slice regular function
defined on a symmetric domain, at any point, are given by
$$
\begin{cases}
s_{2k}=\frac{1}{k!}((\de_{c}\de_{s})^{k}f)(q_{0}),\\
s_{2k+1}=\frac{1}{k!}(\de_{s}(\de_{c}\de_{s})^{k}f)(q_{0}),
\end{cases}
$$
obtaining the first analogy with the complex case.
Moreover we also show explicit computations of some interesting example
involving also idempotent functions and slice-polynomial functions. 
Among other results, we obtain a corollary indicating how to check if a zero of a slice regular function is spherical and what is its spherical multiplicity.

The work done in the previous sections allows us to compare the spherical coefficients of a slice regular function $f$ with those of its first slice derivative $\de_{c}f$
(see Theorem~\ref{coeffslder}), obtaining a non-trivial formula generalizing the equality $\de_{z}^{k}(\de_{z}f)=\de_{z}^{k+1}f$ valid for holomorphic functions. This is 
the content of Section~\ref{slicederi}, where, in Theorem~\ref{infinite} we also obtain as byproduct a countable family of differential equations satisfied by any slice regular function,
the first of which was already obtained in~\cite{altavilladiff} and reads as
$$(\de_{c}f)(x)=(\de_{s}f)(x)+2\Im(x)((\de_{c}\de_{s})f)(x).$$
This sections ends with several other explicit examples and with a small observation on the lack of harmonicity properties for the functions 
$\de_{s}(\de_{c}\de_{s})^{k}f$ for $k\geq 1$. 

After an appendix with the aforementioned flow chart of the results in Section~\ref{alternate} we added another appendix with an alternative proof of Theorem~\ref{coeffslder} given by means of computations analogous to those performed in Section~\ref{alternate}. As in the other technical section, we collect the main results in a table.

\section{Preliminaries}
We recall that, for any slice function $f$ defined on a symmetric domain we have the following equalities 
\begin{equation}\label{basicslice}
\partial_{s}^{2}f\equiv 0,\quad \partial_{s}v_{s}f\equiv 0,\quad v_{s}\partial_{s}f=\partial_{s}f,\quad v_{s}^{2}f=v_{s}f.
\end{equation}
For the first two equalities see~\cite[p. 1671]{G-P}, while the third and fourth equalities are true because
the spherical derivative and the spherical value of a slice function are slice functions induced by stem functions with zero $\sqrt{-1}$-part (see Definition~\ref{sphericalderivativevalue}).

Moreover, if $f$ and $g$ are slice functions on a symmetric domain, then~\cite[Propositions 1.40 and 1.41]{G-S-St}, \cite[Section 5 and Proposition 11]{G-P}
\begin{equation}\label{leibniz}
\partial_{s}(f\cdot g)=\partial_{s}f\cdot v_{s}g+v_{s}f\cdot \partial_{s}g,\qquad \de_{c}(f\cdot g)=\de_{c}f\cdot g+f\cdot \de_{c}g.
\end{equation}
If we add the regularity hypothesis, we have the following immediate result.

\begin{lemma}\label{derivativesphvalue}
Let $f$ be a slice regular function defined on a symmetric domain. Then 
$$\de_{c}v_{s}f=\frac{1}{2}\de_{c}f.$$
\end{lemma}
\begin{proof}
The proof is quite straightforward and relies on the following computation. 
If $f=\mathcal{I}(F_{0}+\sqrt{-1}F_{1})$, then $v_{s}f=\mathcal{I}(F_{0})$. Therefore, if $x=\alpha+I\beta$
\begin{align*}
\left(\de_{c}v_{s}f\right)(x)&=\left(\frac{1}{2}\left(\frac{\partial}{\partial \alpha}-I\frac{\partial}{\partial\beta}\right)F_{0}(\alpha,\beta)\right)\\
&=\frac{1}{2}\left(\frac{1}{2}\left(\frac{\partial}{\partial\alpha}F_{0}(\alpha,\beta)+\frac{\partial}{\partial\beta}F_{1}(\alpha,\beta)\right)+\frac{I}{2}\left(\frac{\partial}{\partial \alpha}F_{1}(\alpha,\beta)-\frac{\partial}{\partial\beta}F_{0}(\alpha,\beta)\right)\right)\\
&=\frac{1}{2}\left(\frac{1}{2}\left(\frac{\partial}{\partial\alpha}-I\frac{\partial}{\partial\beta}\right)(F_{0}(\alpha,\beta)+IF_{1}(\alpha,\beta))\right)\\
&=\frac{1}{2}\de_{c}f(x).
\end{align*}
\end{proof}
Thanks to the previous lemma we are able to compute the value of the second order differential operator $\de_{c}\partial_{s}$, applied to the product of two slice regular functions
and even to understand its behaviors under following iterations.

\begin{lemma}\label{firstiteration}
Let $f$ and $g$ be two slice regular functions defined on a symmetric domain $\Omega_{D}$. Then for any $x\in\Omega_{D}\setminus\R$ we have
$$
\left(\left(\de_{c}\de_{s}\right)(f\cdot g)\right)(x)=(\left(\de_{c}\de_{s}f\right)\cdot v_{s}g)(x)+\frac{1}{2}\left[(\de_{s}f\cdot \de_{c}g)(x)+(\de_{c}f\cdot \de_{s}g)(x)\right]+(v_{s}f\cdot \left(\de_{c}\de_{s}g\right))(x).
$$
\end{lemma}
\begin{proof}
The proof is obtained thanks to the equalities in Formula~\eqref{leibniz} and to Lemma~\ref{derivativesphvalue}.
\end{proof}
By~\cite[Theorem 6.3]{Perotti}, if $f$ is a slice regular function, then $\partial_{c}\partial_{s}f$
equals $-1/4$ the four dimensional real Laplacian of $f$. From this point of view the formula in the previous lemma is related to the standard
formula for the Laplacian of a product.

Recalling now the equalities given in Formula~\eqref{basicslice}, we are able to compute the next iterations of the operator $\de_{c}\de_{s}$ applied to 
the product of two slice regular functions. For instance we have the following result
\begin{lemma}\label{lemmanuovo}
Let $f$ and $g$ be two slice regular functions defined on a symmetric domain $\Omega_{D}$. Then for any $x\in\Omega_{D}\setminus\R$ we have\begin{multline}\label{step2}
\left(\de_{c}\de_{s}\right)^{2}(f\cdot g)=\left(\left(\de_{c}\de_{s}\right)^{2}f\right)\cdot v_{s}g+v_{s}f\cdot \left(\left(\de_{c}\de_{s}\right)^{2}g\right)\\
+\frac{1}{2}\Big[\de_{s}\left(\de_{c}\de_{s}\right)f\cdot \de_{c}g+
\left(\left(\de_{c}\de_{s}\right)f\right)\cdot \left(\de_{s}\de_{c}g\right)+
\de_{s}f\cdot \left(\de_{c}\left(\de_{s}\de_{c}g\right)\right)\\
+\left(\de_{c}\left(\de_{s}\de_{c}f\right)\right)\cdot \de_{s}g+
\left(\de_{s}\de_{c}f\right)\cdot \left(\de_{c}\de_{s}g\right)+
\de_{c}f\cdot \left(\de_{s}\left(\de_{c}\de_{s}g\right)\right)
\Big],
\end{multline}
where in the equality we have omitted the evaluation at $x\in\Omega_{D}\setminus\R$.
\end{lemma}
\begin{proof}
By applying the spherical derivative operator to the terms in the equality in Lemma~\ref{firstiteration}, thanks to  Formulas~\eqref{basicslice} and~\eqref{leibniz}, we get
\begin{align*}
\de_{s}\left(\left(\de_{c}\de_{s}\right)(f\cdot g)\right)&=
\de_{s}(\left(\de_{c}\de_{s}f\right)\cdot v_{s}g)+\frac{1}{2}\de_{s}\left[(\de_{s}f\cdot \de_{c}g)+(\de_{c}f\cdot \de_{s}g)\right]+\de_{s}(v_{s}f\cdot \left(\de_{c}\de_{s}g\right))\\
&=(\left(\de_{s}(\de_{c}\de_{s})f\right)\cdot v_{s}g)+\frac{1}{2}\left[(\de_{s}f\cdot (\de_{s}\de_{c}g))+((\de_{s}\de_{c}f)\cdot \de_{s}g)\right]+(v_{s}f\cdot (\de_{s}\left(\de_{c}\de_{s}g\right))).
\end{align*}
Again, thanks to equalities in Formulas~\ref{basicslice} and~\ref{leibniz} we obtain
\begin{align*}
\left(\left(\de_{c}\de_{s}\right)^{2}(f\cdot g)\right)
&=\de_{c}(\left(\de_{s}(\de_{c}\de_{s})f\right)\cdot v_{s}g)+\frac{1}{2}\de_{c}\left[(\de_{s}f\cdot (\de_{s}\de_{c}g))+((\de_{s}\de_{c}f)\cdot \de_{s}g)\right]+\de_{c}(v_{s}f\cdot (\de_{s}\left(\de_{c}\de_{s}g\right)))\\
&=\left(\left(\de_{c}\de_{s}\right)^{2}f\right)\cdot v_{s}g+\left(\de_{s}(\de_{c}\de_{s})f\right)\cdot \frac{1}{2}\de_{c}g +\frac{1}{2}\de_{c}f\cdot (\de_{s}\left(\de_{c}\de_{s}g\right)))+v_{s}f\cdot \left(\left(\de_{c}\de_{s}\right)^{2}g\right)+\\
&\,\,\,\,\,\,\,\,\,+\frac{1}{2}\left[((\de_{c}\de_{s}f)\cdot (\de_{s}\de_{c}g))+(\de_{s}f\cdot \de_{c}(\de_{s}\de_{c}g))+(\de_{c}(\de_{s}\de_{c}f)\cdot \de_{s}g)+((\de_{s}\de_{c}f)\cdot (\de_{c}\de_{s}g))\right],
\end{align*}
which, up to a reordering of the terms, gives the thesis.
\end{proof}

After a couple of others iteration one realizes that a non-trivial pattern arises. The following main result exposes such pattern.
\begin{proposition}\label{main}
Let $f$ and $g$ be two slice regular functions defined on a symmetric domain, then, for any $k\in\mathbb{N}\setminus\{0\}$, we have
\begin{align*}
\left(\de_{c}\de_{s}\right)^{k}(f\cdot g)&=\left(\left(\de_{c}\de_{s}\right)^{k}f\right)\cdot v_{s}g+v_{s}f\cdot \left(\left(\de_{c}\de_{s}\right)^{k}g\right)\\
&+\frac{1}{2}\left\{\sum_{h=0}^{k-1}\binom{k-1}{h}\left[(\de_{s}(\de_{c}\de_{s})^{k-(h+1)}f)\cdot (\de_{c}(\de_{s}\de_{c})^{h}g)+(\de_{c}(\de_{s}\de_{c})^{h}f)\cdot (\de_{s}(\de_{c}\de_{s})^{k-(h+1)}g)\right]\right.\\
&\,\,\,\,\,\,\,\,\,\,\left.+\sum_{h=1}^{k-1}\binom{k-1}{h}\left[((\de_{c}\de_{s})^{k-h}f)\cdot ((\de_{s}\de_{c})^{h}g)+((\de_{s}\de_{c})^{h}f)\cdot ((\de_{c}\de_{s})^{k-h}g)\right]\right\}.
\end{align*}
\end{proposition}
\begin{proof}
The proof is done by induction. The base cases ($k=1,2$) are given in Lemmas~\ref{firstiteration} and~\ref{lemmanuovo}. Assume now that the equality is true for
$k\in\mathbb{N}$. We will prove it for $k+1$. Firstly, by spending immediately the inductive hypothesis, we have 
\begin{align*}
\left(\de_{c}\de_{s}\right)^{k+1}(f\cdot g)&=(\de_{c}\de_{s})(\left(\de_{c}\de_{s}\right)^{k}(f\cdot g))\\
&=(\de_{c}\de_{s})\left\{\left(\left(\de_{c}\de_{s}\right)^{k}f\right)\cdot v_{s}g+v_{s}f\cdot \left(\left(\de_{c}\de_{s}\right)^{k}g\right)\right.\\
&+\frac{1}{2}\left\{\sum_{h=0}^{k-1}\binom{k-1}{h}\left[(\de_{s}(\de_{c}\de_{s})^{k-(h+1)}f)\cdot (\de_{c}(\de_{s}\de_{c})^{h}g)+(\de_{c}(\de_{s}\de_{c})^{h}f)\cdot (\de_{s}(\de_{c}\de_{s})^{k-(h+1)}g)\right]\right.\\
&\left.\,\,\,\,\,\,\,\,\,\,\left.+\sum_{h=1}^{k-1}\binom{k-1}{h}\left[((\de_{c}\de_{s})^{k-h}f)\cdot ((\de_{s}\de_{c})^{h}g)+((\de_{s}\de_{c})^{h}f)\cdot ((\de_{c}\de_{s})^{k-h}g)\right]\right\}\right\}.
\end{align*}
Let us apply the spherical derivative $\partial_{s}$, taking into account formulas~\eqref{basicslice}: we obtain
\begin{align*}
\left(\de_{c}\de_{s}\right)^{k+1}(f\cdot g)&=\de_{c}\left\{\left(\de_{s}\left(\de_{c}\de_{s}\right)^{k}f\right)\cdot v_{s}g+v_{s}f\cdot \left(\de_{s}\left(\de_{c}\de_{s}\right)^{k}g\right)\right.\\
&+\frac{1}{2}\left\{\sum_{h=0}^{k-1}\binom{k-1}{h}\left[(\de_{s}(\de_{c}\de_{s})^{k-(h+1)}f)\cdot ((\de_{s}\de_{c})^{h+1}g)+((\de_{s}\de_{c})^{h+1}f)\cdot (\de_{s}(\de_{c}\de_{s})^{k-(h+1)}g)\right]\right.\\
&\left.\,\,\,\,\,\,\,\,\,\,\left.+\sum_{h=1}^{k-1}\binom{k-1}{h}\left[(\de_{s}(\de_{c}\de_{s})^{k-h}f)\cdot ((\de_{s}\de_{c})^{h}g)+((\de_{s}\de_{c})^{h}f)\cdot (\de_{s}(\de_{c}\de_{s})^{k-h}g)\right]\right\}\right\}.
\end{align*}
Using now the Leibniz rule for $\de_{c}$ and Lemma~\ref{derivativesphvalue} we get
\begin{align*}
\left(\de_{c}\de_{s}\right)^{k+1}(f\cdot g)&=\left(\left(\de_{c}\de_{s}\right)^{k+1}f\right)\cdot v_{s}g+v_{s}f\cdot \left(\left(\de_{c}\de_{s}\right)^{k+1}g\right)\\
&+\frac{1}{2}\left\{\left(\de_{s}\left(\de_{c}\de_{s}\right)^{k}f\right)\cdot \de_{c}g+\de_{c}f\cdot \left(\de_{s}\left(\de_{c}\de_{s}\right)^{k}g\right) \right.\\
&+\sum_{h=0}^{k-1}\binom{k-1}{h}\left[
((\de_{c}\de_{s})^{k-h}f)\cdot ((\de_{s}\de_{c})^{h+1}g)+(\de_{s}(\de_{c}\de_{s})^{k-(h+1)}f)\cdot (\de_{c}(\de_{s}\de_{c})^{h+1}g)\right.\\
&\,\,\,\,\,\,\,\,\,\,\,\,\,\,\,\,\,\,\,\,\,\,\,\,\,\,\,\,\,\,\,\,\,\,\,\,\,\,\,\,+\left.(\de_{c}(\de_{s}\de_{c})^{h+1}f)\cdot (\de_{s}(\de_{c}\de_{s})^{k-(h+1)}g)+((\de_{s}\de_{c})^{h+1}f)\cdot ((\de_{c}\de_{s})^{k-h}g)\right]\\
&+\sum_{h=1}^{k-1}\binom{k-1}{h}\left[((\de_{c}\de_{s})^{k-h+1}f)\cdot ((\de_{s}\de_{c})^{h}g)+(\de_{s}(\de_{c}\de_{s})^{k-h}f)\cdot (\de_{c}(\de_{s}\de_{c})^{h}g)\right.\\
&\,\,\,\,\,\,\,\,\,\,\,\,\,\,\,\,\,\,\,\,\,\,\,\,\,\,\,\,\,\,\,\,\,\,\,\,\,\,\,\,\left.+(\de_{c}(\de_{s}\de_{c})^{h}f)\cdot (\de_{s}(\de_{c}\de_{s})^{k-h}g)+((\de_{s}\de_{c})^{h}f)\cdot ((\de_{c}\de_{s})^{k-h+1}g)\right]\Big\}.
\end{align*}
We now reorder the previous sum in a more convenient way as follows:
\begin{align*}
\left(\de_{c}\de_{s}\right)^{k+1}(f\cdot g)&=\left(\left(\de_{c}\de_{s}\right)^{k+1}f\right)\cdot v_{s}g+v_{s}f\cdot \left(\left(\de_{c}\de_{s}\right)^{k+1}g\right)\\
&+\frac{1}{2}\Big\{\left(\de_{s}\left(\de_{c}\de_{s}\right)^{k}f\right)\cdot \de_{c}g+\de_{c}f\cdot \left(\de_{s}\left(\de_{c}\de_{s}\right)^{k}g\right)\\
&+\sum_{h=0}^{k-1}\binom{k-1}{h}\left[(\de_{s}(\de_{c}\de_{s})^{k-(h+1)}f)\cdot (\de_{c}(\de_{s}\de_{c})^{h+1}g)+(\de_{c}(\de_{s}\de_{c})^{h+1}f)\cdot (\de_{s}(\de_{c}\de_{s})^{k-(h+1)}g)\right]\\
&+\sum_{h=1}^{k-1}\binom{k-1}{h}\left[(\de_{s}(\de_{c}\de_{s})^{k-h}f)\cdot (\de_{c}(\de_{s}\de_{c})^{h}g)+(\de_{c}(\de_{s}\de_{c})^{h}f)\cdot (\de_{s}(\de_{c}\de_{s})^{k-h}g)\right]\\
&+\sum_{h=0}^{k-1}\binom{k-1}{h}\left[((\de_{c}\de_{s})^{k-h}f)\cdot ((\de_{s}\de_{c})^{h+1}g)+((\de_{s}\de_{c})^{h+1}f)\cdot ((\de_{c}\de_{s})^{k-h}g)\right]\\
&+\sum_{h=1}^{k-1}\binom{k-1}{h}(\left[(\de_{c}\de_{s})^{k-h+1}f)\cdot ((\de_{s}\de_{c})^{h}g)+((\de_{s}\de_{c})^{h}f)\cdot ((\de_{c}\de_{s})^{k-h+1}g)\right]\Big\}.
\end{align*}
If we set $h=p-1$ in the two summations starting from $h=0$, we get
\begin{align*}
\left(\de_{c}\de_{s}\right)^{k+1}(f\cdot g)&=\left(\left(\de_{c}\de_{s}\right)^{k+1}f\right)\cdot v_{s}g+v_{s}f\cdot \left(\left(\de_{c}\de_{s}\right)^{k+1}g\right)\\
&+\frac{1}{2}\Big\{\left(\de_{s}\left(\de_{c}\de_{s}\right)^{k}f\right)\cdot \de_{c}g+\de_{c}f\cdot \left(\de_{s}\left(\de_{c}\de_{s}\right)^{k}g\right)\\
&+\sum_{p=1}^{k}\binom{k-1}{p-1}\left[(\de_{s}(\de_{c}\de_{s})^{k-p}f)\cdot (\de_{c}(\de_{s}\de_{c})^{p}g)+(\de_{c}(\de_{s}\de_{c})^{p}f)\cdot (\de_{s}(\de_{c}\de_{s})^{k-p}g)\right]\\
&+\sum_{h=1}^{k-1}\binom{k-1}{h}\left[(\de_{s}(\de_{c}\de_{s})^{k-h}f)\cdot (\de_{c}(\de_{s}\de_{c})^{h}g)+(\de_{c}(\de_{s}\de_{c})^{h}f)\cdot (\de_{s}(\de_{c}\de_{s})^{k-h}g)\right]\\
&+\sum_{p=1}^{k}\binom{k-1}{p-1}\left[((\de_{c}\de_{s})^{k-p+1}f)\cdot ((\de_{s}\de_{c})^{p}g)+((\de_{s}\de_{c})^{p}f)\cdot ((\de_{c}\de_{s})^{k-p+1}g)\right]\\
&+\sum_{h=1}^{k-1}\binom{k-1}{h}\left[((\de_{c}\de_{s})^{k-h+1}f)\cdot ((\de_{s}\de_{c})^{h}g)+((\de_{s}\de_{c})^{h}f)\cdot ((\de_{c}\de_{s})^{k-h+1}g)\right]\Big\}.
\end{align*}
Now let us substitute $p$ for $h$. We recall that, for all $p\in\{1,\ldots,k-1\}$, it holds
$$
\binom{k-1}{p-1}+\binom{k-1}{p}=\binom{k}{p},
$$
and $\binom{k-1}{k-1}=1=\binom{k}{k}$. Thus, we get the thesis
\begin{align*}
\left(\de_{c}\de_{s}\right)^{k+1}(f\cdot g)&=\left(\left(\de_{c}\de_{s}\right)^{k+1}f\right)\cdot v_{s}g+v_{s}f\cdot \left(\left(\de_{c}\de_{s}\right)^{k+1}g\right)\\
&+\frac{1}{2}\left\{\sum_{p=0}^{k}\binom{k}{p}\left[(\de_{s}(\de_{c}\de_{s})^{k-p}f)\cdot (\de_{c}(\de_{s}\de_{c})^{p}g)+(\de_{c}(\de_{s}\de_{c})^{p}f)\cdot (\de_{s}(\de_{c}\de_{s})^{k-p}g)\right]\right.\\
&\,\,\,\,\,\,\,\,\,\,\left.+\sum_{p=1}^{k}\binom{k}{p}\left[((\de_{c}\de_{s})^{(k+1)-p}f)\cdot ((\de_{s}\de_{c})^{p}g)+((\de_{s}\de_{c})^{p}f)\cdot ((\de_{c}\de_{s})^{(k+1)-p}g)\right]\right\}.
\end{align*}
\end{proof}

\begin{remark}\label{lowering}
Let $f=\mathcal{I}(F_{0}+\sqrt{-1}F_{1})$ be any slice function such that its components $F_{0}$ and $F_{1}$ are homogeneous polynomials of degree $k$ in $\alpha$ and $\beta$.
Then, clearly, $\de_{s}f$ and $\de_{c}f$ are slice functions such that their respective stem functions are homogeneous polynomials of degree $k-1$.
\end{remark}
Thanks to the previous remark, we can state the following Lemma.
\begin{lemma}\label{lemmapowers}
Let $n \in \mathbb{N}$. We have the following equalities.
\begin{itemize}
\item If $n>0$, then 
$$
\de_{s}(\de_{c}\de_{s})^{n}x^{2n}=\de_{c}(\de_{s}\de_{c})^{n}x^{2n}\equiv 0.
$$
\item If $n\geq 0$, then 
$$
(\de_{c}\de_{s})^{n+1}x^{2n+1}=(\de_{s}\de_{c})^{n+1}x^{2n+1}\equiv 0.
$$
\end{itemize}
\end{lemma}
\begin{proof}
Thanks to Remark~\ref{lowering}, by applying to $x^{k}$ an alternate sequence of the operators $\de_{s}$ and $\de_{c}$, it turns out that the $k$-th iteration returns a constant
and hence the $(k+1)$-th returns zero.
\end{proof}

Before stating the following corollary we recall from Definition~\ref{characteristicpolynomial} that, for any $q_{0}\in\HH\setminus\R$ the characteristic
polynomial of $q_{0}$ is the quadratic quaternionic polynomial given by $\Delta_{q_{0}}(x)=(x-q_{0})\cdot(x-q_{0}^{c})$.

\begin{corollary}\label{cor1}
Let $q_{0}\in\HH\setminus\R$ be any quaternion and $x=\alpha+I\beta$. Then we have the following equalities.
\begin{enumerate}
\item 
$$
\de_{s}(x-q_{0})=\de_{c}(x-q_{0})\equiv1,\qquad v_{s}(x-q_{0})=\alpha-q_{0},\qquad (\de_{s}\de_{c})(x-q_{0})=(\de_{c}\de_{s})(x-q_{0})\equiv 0.
$$
\item 
$$\de_{s}\Delta_{q_{0}}=2(\alpha-\Re(q_{0})),\qquad\de_{c}\Delta_{q_{0}}=2(x-\Re(q_{0})),\qquad v_{s}\Delta_{q_{0}}=\alpha^{2}-\beta^{2}-2\alpha \Re(q_{0})+|q_{0}|^{2}.$$
Notice that $(\de_{s}\Delta_{q_{0}})(q_{0})=(v_{s}\Delta_{q_{0}})(q_{0})=0$. Moreover
$$
(\de_{c}\de_{s})\Delta_{q_{0}}\equiv 1,\qquad (\de_{s}\de_{c})\Delta_{q_{0}}\equiv 2.
$$
\item For any integer $k\geq 0$, we have 
$$
\de_{s}(\de_{c}\de_{s})^{k}\Delta_{q_{0}}^{k}=\de_{c}(\de_{s}\de_{c})^{k}\Delta_{q_{0}}^{k}\equiv 0,\qquad 
(\de_{c}\de_{s})^{k+1}(\Delta^{k}_{q_{0}}(x) (x-q_{0}))= (\de_{s}\de_{c})^{k+1}(\Delta^{k}_{q_{0}}(x) (x-q_{0}))\equiv 0.
$$
\end{enumerate}

\end{corollary}

\begin{proof}
Points $(1)$ and $(2)$ are straightforward consequences of the definitions of $\Delta_{q_{0}}$ and of the differential operators $\de_{s}$ and $\de_{c}$. For point 
$(3)$, just notice that $\Delta^{k}_{q_{0}}$ is a quaternionic polynomial of degree $2k$, while $\Delta^{k}_{q_{0}}(x) (x-q_{0})$ is a quaternionic polynomial of degree $2k+1$.
Hence, by Lemma~\ref{lemmapowers}, we get the thesis.
\end{proof}

\section{Alternate actions of $\de_{s}$ and $\de_{c}$ on a characteristic polynomial}\label{alternate}
In this section we will give a series of technical results needed to prove our first main theorem. Notice that the order of the result contained in this section is not completely arbitrary. In fact as the reader might notice, all such results are somehow linked one to another
as showed in the flow chart given in Appendix~\ref{flow}.

We start from the following lemma.

\begin{lemma}\label{lemma2}
Let $q_{0}\in\HH$ and $p\in\mathbb{N}\setminus\{0\}$. Then, for any $0\leq k<p$ (and so for any $k\in\mathbb{\mathbb{N}}$) we have
$$
(\de_{s}(\de_{c}\de_{s})^{k}\Delta_{q_{0}}^{p})(q_{0})=0.
$$
\end{lemma}
\begin{proof}
First of all, for $k=0$, the thesis follows from Corollary~\ref{cor1}.
Assume now that $k>0$ and $p>k$. By Proposition~\ref{main} and Corollary~\ref{cor1}, we have
\begin{align*}
((\de_{s}(\de_{c}\de_{s})^{k})\Delta_{q_{0}}^{p})&=((\de_{s}(\de_{c}\de_{s})^{k})(\Delta_{q_{0}}^{p-1}\Delta_{q_{0}}))\\
&=\de_{s}\bigg\{((\de_{c}\de_{s})^{k})\Delta_{q_{0}}^{p-1})\cdot v_{s}\Delta_{q_{0}}+v_{s}\Delta_{q_{0}}^{p-1}\cdot ((\de_{c}\de_{s})^{k}\Delta_{q_{0}})\\
&\quad\quad+\frac{1}{2}\left[\binom{k-1}{0}(\de_{s}(\de_{c}\de_{s})^{k-1}\Delta_{q_{0}}^{p-1})\cdot \de_{c}\Delta_{q_{0}}+\binom{k-1}{k-1}(\de_{c}(\de_{s}\de_{c})^{k-1}\Delta_{q_{0}}^{p-1})\cdot \de_{s}\Delta_{q_{0}}\right.\\
&\quad\quad\,\,\,\quad\left.\left.
+\binom{k-1}{1}((\de_{c}\de_{s})^{k-1}\Delta_{q_{0}}^{p-1})\cdot 2+\binom{k-1}{k-1}((\de_{s}\de_{c})^{k-1}\Delta_{q_{0}}^{p-1})\right]\right\}\\
&=(\de_{s}(\de_{c}\de_{s})^{k})\Delta_{q_{0}}^{p-1})\cdot v_{s}\Delta_{q_{0}}+v_{s}\Delta_{q_{0}}^{p-1}\cdot (\de_{s}(\de_{c}\de_{s})^{k}\Delta_{q_{0}})\\
&\quad\quad+\frac{1}{2}\left[(\de_{s}(\de_{c}\de_{s})^{k-1}\Delta_{q_{0}}^{p-1})\cdot (\de_{s}\de_{c})\Delta_{q_{0}}+((\de_{s}\de_{c})^{k}\Delta_{q_{0}}^{p-1})\cdot \de_{s}\Delta_{q_{0}}\right.\\
&\quad\quad\left.
+(k-1)(\de_{s}(\de_{c}\de_{s})^{k-1}\Delta_{q_{0}}^{p-1})\cdot 2\right].
\end{align*}
Now, by using again the fact that $(\de_{s}\de_{c})\Delta_{q_{0}}=2$ and that, for any $m\in\mathbb{N}$, $(v_{s}\Delta_{q_{0}}^{m})(q_{0})=(\de_{s}\Delta_{q_{0}}^{m})(q_{0})=0$, for $x=q_{0}$, we obtain
$$
((\de_{s}(\de_{c}\de_{s})^{k})\Delta_{q_{0}}^{p})(q_{0})=k(\de_{s}(\de_{c}\de_{s})^{k-1}\Delta_{q_{0}}^{p-1})(q_{0}).
$$
Iterating the previous computations and argument $k$ times, we get
$$
((\de_{s}(\de_{c}\de_{s})^{k})\Delta_{q_{0}}^{p})(q_{0})=k\cdot(k-1)\cdot\ldots\cdot2\cdot1(\de_{s}\Delta_{q_{0}}^{p-k})(q_{0})
$$
but, as $p>k$, by the previous case $k=0$, we have $((\de_{s}(\de_{c}\de_{s})^{k})\Delta_{q_{0}}^{p})(q_{0})=0$. 
\end{proof}

We now evaluate at $q_{0}$ the $k$-th iteration of the operator $(\de_{s}\de_{c})$ applied to the $p$-th power of the characteristic polynomial $\Delta_{q_{0}}$.

\begin{lemma}\label{lemma5}
Let $q_{0}\in\HH$ and $p\in\mathbb{N}\setminus\{0\}$. Then, for any $0\leq k\leq p$ we have
$$
((\de_{s}\de_{c})^{k}\Delta_{q_{0}}^{p})(q_{0})=
\begin{cases}
2\cdot k!,\qquad\mbox{ if }k=p,\\
0,\qquad{ otherwise.}
\end{cases}
$$
\end{lemma}

\begin{proof}
If $k=0$ the result is obvious, since $\Delta_{q_{0}}^{p}(q_{0})=0$. If $k=1<p$ we have 
$$
((\de_{s}\de_{c})\Delta_{q_{0}}^{p})(x)=\de_{s}(2p\Delta_{q_{0}}^{p-1}(x)(x-\Re(q_{0})))=2p[(\de_{s}\Delta_{q_{0}}^{p-1})\cdot (v_{s}(x-\Re(q_{0})))+(v_{s}\Delta_{q_{0}}^{p-1})\cdot (\de_{s}(x-\Re(q_{0})))],
$$
but again, the previous expression vanishes when evaluated at $x=q_{0}$ as $(\de_{s}\Delta_{q_{0}}^{p-1})(q_{0})=(v_{s}\Delta_{q_{0}}^{p-1})(q_{0})=0$.
Assume now that $k>1$ and that $p>k$. By Proposition~\ref{main} and Corollary~\ref{cor1}, we have
\begin{align*}
(\de_{s}\de_{c})^{k}\Delta_{q_{0}}^{p}&=\de_{s}((\de_{c}\de_{s})^{k-1}(2p\Delta_{q_{0}}^{p-1}(x)(x-\Re(q_{0}))))=2p\de_{s}\bigg\{((\de_{c}\de_{s})^{k-1}\Delta_{q_{0}}^{p-1}(x))\cdot v_{s}(x-\Re(q_{0}))\\
&\quad\left.+\frac{1}{2}\left[\binom{k-2}{0}\de_{s}(\de_{c}\de_{s})^{k-2}\Delta_{q_{0}}^{p-1}(x)+\binom{k-2}{k-2}\de_{c}(\de_{s}\de_{c})^{k-2}\Delta_{q_{0}}^{p-1}(x)\right]\right\}\\
&=2p\left\{(\de_{s}(\de_{c}\de_{s})^{k-1}\Delta_{q_{0}}^{p-1}(x))\cdot v_{s}(x-\Re(q_{0}))+\frac{1}{2}(\de_{s}\de_{c})^{k-1}\Delta_{q_{0}}^{p-1}(x)\right\},
\end{align*}
but since $(v_{s}(x-\Re(q_{0})))_{|_{x=q_{0}}}=0$, then we have
$$
((\de_{s}\de_{c})^{k}\Delta_{q_{0}}^{p})(q_{0})=p((\de_{s}\de_{c})^{k-1}\Delta_{q_{0}}^{p-1})(q_{0}).
$$
Iterating the same computations and argument $k$ times, we get
$$
((\de_{s}\de_{c})^{k}\Delta_{q_{0}}^{p})(q_{0})=p\cdot(p-1)\cdot\ldots\cdot2\cdot1((\de_{s}\de_{c})\Delta_{q_{0}}^{p-k})(q_{0}),
$$
and, thanks to the previous computation, we get the thesis.
If instead $p=k$, we proceed by induction.

The base case, for $k=1$ is done in Corollary~\ref{cor1}, point $(2)$.
Assume now that the result is true for $k-1\in\mathbb{N}$ and set $x=\alpha+I\beta$. Then,
\begin{align*}
((\de_{s}\de_{c})^{k}\Delta_{q_{0}}^{k})(x)&=((\de_{s}(\de_{c}\de_{s})^{k-1}\de_{c})\Delta_{q_{0}}^{k})(x)=(\de_{s}(\de_{c}\de_{s})^{k-1})(k\Delta_{q_{0}}^{k-1}(x)((x-q_{0})+(x-\bar q_{0})))\\
&=2k(\de_{s}(\de_{c}\de_{s})^{k-1})(\Delta_{q_{0}}^{k-1}(x)(x-\Re(q_{0}))).
\end{align*}
We now apply the formula contained in Proposition~\ref{main}, and, recalling the equalities in Lemma~\ref{lemmapowers} and Corollary~\ref{cor1}, obtain
\begin{align*}
((\de_{s}\de_{c})^{k}\Delta_{q_{0}}^{k})(x)&=2k(\de_{s}(\de_{c}\de_{s})^{k-1})(\Delta_{q_{0}}^{k-1}(x)(x-\Re(q_{0})))=2k\de_{s}\bigg\{(\de_{c}\de_{s})^{k-1}(\Delta_{q_{0}}^{k-1}(x)(\alpha-\Re(q_{0})))\\
&\left.\quad+\frac{1}{2}\left[\binom{k-2}{0}(\de_{s}(\de_{c}\de_{s})^{k-2})\Delta_{q_{0}}^{k-1}(x)+
\binom{k-2}{k-2}(\de_{c}(\de_{s}\de_{c})^{k-2})\Delta_{q_{0}}^{k-1}(x)\right]\right\}\\
&=2k\left\{\frac{1}{2}(\de_{s}\de_{c})^{k-1}\Delta_{q_{0}}^{k-1}(x)\right\}=k\cdot 2\cdot(k-1)!=2\cdot k!,
\end{align*}
where in the last line of the previous sequence of computations, we have used the inductive hypothesis.
\end{proof}

%

Thanks to the previous results we are now able to prove our first main result.

\begin{theorem}\label{lemma6}
Let $q_{0}\in\HH$ and $p\in\mathbb{N}\setminus\{0\}$. Then, for any $0\leq k\leq p$ we have
$$
((\de_{c}\de_{s})^{k}\Delta_{q_{0}}^{p})(q_{0})=
\begin{cases}
k!,\qquad\mbox{ if }k=p,\\
0\qquad\mbox{otherwise.}
\end{cases}
$$
\end{theorem}
\begin{proof}
If $k=0$ the result is obvious, since $(\Delta_{q_{0}}^{p})(q_{0})=0$. If $k=1<p$ we have 
\begin{align*}
(\de_{c}\de_{s})\Delta_{q_{0}}^{p}&=\de_{c}[(\de_{s}\Delta_{q_{0}}^{p-1})\cdot (v_{s}\Delta_{q_{0}})+(v_{s}\Delta_{q_{0}}^{p-1})\cdot (\de_{s}\Delta_{q_{0}})]\\
&=((\de_{c}\de_{s})\Delta_{q_{0}}^{p-1})\cdot (v_{s}\Delta_{q_{0}})+(v_{s}\Delta_{q_{0}}^{p-1})\cdot (\de_{c}\de_{s})\Delta_{q_{0}}+\frac{1}{2}\left[(\de_{s}\Delta_{q_{0}}^{p-1})\cdot (\de_{c}\Delta_{q_{0}})+(\de_{c}\Delta_{q_{0}}^{p-1})\cdot (\de_{s}\Delta_{q_{0}})
\right],
\end{align*}
but again, the previous expression vanishes when evaluated at $x=q_{0}$ as $(\de_{s}\Delta_{q_{0}}^{p-1})(q_{0})=(v_{s}\Delta_{q_{0}}^{p-1})(q_{0})=0$.
Assume now that $k>1$ and that $p>k$. By Proposition~\ref{main}, we have
\begin{align*}
(\de_{c}\de_{s})^{k}\Delta_{q_{0}}^{p}&=(\de_{c}\de_{s})^{k}(\Delta_{q_{0}}^{p-1}\Delta_{q_{0}})\\
&=((\de_{c}\de_{s})^{k}\Delta_{q_{0}}^{p-1})\cdot v_{s}\Delta_{q_{0}}+(v_{s}\Delta_{q_{0}}^{p-1})\cdot ((\de_{c}\de_{s})^{k}\Delta_{q_{0}})\\
&+\frac{1}{2}\left\{\binom{k-1}{0}(\de_{s}(\de_{c}\de_{s})^{k-1}\Delta_{q_{0}}^{p-1})\cdot \de_{c}\Delta_{q_{0}}+\binom{k-1}{k-1}(\de_{c}(\de_{s}\de_{c})^{k-1}\Delta_{q_{0}}^{p-1})\cdot \de_{s}\Delta_{q_{0}}\right.\\
&\left.\quad\quad+\binom{k-1}{1}((\de_{c}\de_{s})^{k-1}\Delta_{q_{0}}^{p-1})\cdot 2+\binom{k-1}{k-1}((\de_{s}\de_{c})^{k-1}\Delta_{q_{0}}^{p-1})\right\}.
\end{align*}
Therefore, for $x=q_{0}$, using Lemma~\ref{lemma2}, Lemma~\ref{lemma5} and the fact that $(\de_{s}\Delta_{q_{0}}^{m})(q_{0})=(v_{s}\Delta_{q_{0}}^{m})(q_{0})=0$, we get
$$
((\de_{c}\de_{s})^{k}\Delta_{q_{0}}^{p})(q_{0})=(k-1)((\de_{c}\de_{s})^{k-1}\Delta_{q_{0}}^{p-1})(q_{0}).
$$
Iterating the same computation and argument $k-1$ times, we obtain
$$
((\de_{c}\de_{s})^{k}\Delta_{q_{0}}^{p})(q_{0})=(k-1)\cdot(k-2)\cdot\ldots\cdot2\cdot1((\de_{c}\de_{s})\Delta_{q_{0}}^{p-(k-1)})(q_{0}).
$$
But as $p>k$, thanks to the previous case $k=1$, we get the thesis.
If instead $k=p$, then we proceed by induction. The base case, for $k=1$ is addressed in Corollary~\ref{cor1}, point $(2)$.
Assume now that $k$ is greater or equal than $2$ and that the result is true for $k-1\in\mathbb{N}$ and set $x=\alpha+I\beta$. Thanks to Proposition~\ref{main}, and the fact that $\Delta_{q_{0}}$ has degree $2$, we have,
\begin{align*}
(\de_{c}\de_{s})^{k}\Delta_{q_{0}}^{k}&=(\de_{c}\de_{s})^{k}(\Delta_{q_{0}}^{k-1}\Delta_{q_{0}})\\
&=(\de_{c}\de_{s})^{k}\Delta_{q_{0}}^{k-1}v_{s}\Delta_{q_{0}}+v_{s}\Delta_{q_{0}}^{k-1}(\de_{c}\de_{s})^{k}\Delta_{q_{0}}\\
&\quad+\frac{1}{2}\left\{\binom{k-1}{0}(\de_{s}(\de_{c}\de_{s})^{k-1}\Delta_{q_{0}}^{k-1})\de_{c}\Delta_{q_{0}}+\binom{k-1}{k-1}(\de_{c}(\de_{s}\de_{c})^{k-1}\Delta_{q_{0}}^{k-1})\de_{s}\Delta_{q_{0}}\right.\\
&\left.\quad\quad\quad+\binom{k-1}{1}((\de_{c}\de_{s})^{k-1}\Delta_{q_{0}}^{k-1})((\de_{s}\de_{c})\Delta_{q_{0}})+\binom{k-1}{k-1}((\de_{s}\de_{c})^{k-1}\Delta_{q_{0}}^{k-1})((\de_{c}\de_{s})\Delta_{q_{0}})\right\}.
\end{align*}
Now, thanks to Corollary~\ref{cor1} we have
$$
(\de_{c}\de_{s})^{k}\Delta_{q_{0}}^{k}=\frac{1}{2}\left\{(k-1)((\de_{c}\de_{s})^{k-1}\Delta_{q_{0}}^{k-1})\cdot 2+((\de_{s}\de_{c})^{k-1}\Delta_{q_{0}}^{k-1})\right\}.
$$
Using the inductive hypothesis and Lemma~\ref{lemma5}, we finally have
$$
(\de_{c}\de_{s})^{k}\Delta_{q_{0}}^{k}=\frac{1}{2}\left\{(k-1)(k-1)!\cdot 2+2(k-1)!\right\}=(k-1)![(k-1)+1]=k!,
$$
proving the thesis.
\end{proof}

We now consider a slight modification of Lemma~\ref{lemma5}.
\begin{lemma}\label{lemma3}
Let $q_{0}\in\HH$ and $p\in\mathbb{N}\setminus\{0\}$. Then, for any $0\leq k\leq p-1$ we have
$$
((\de_{c}(\de_{s}\de_c)^{k})\Delta_{q_{0}}^{p})(q_{0})=\begin{cases}
2\Im(q_{0})k!,\qquad\mbox{ if }k=p-1,\\
0,\,\,\,\quad\qquad\qquad\mbox{ otherwise.}
\end{cases}
$$
\end{lemma}
\begin{proof}
For $k=0$ and any $p>1$ we have that $(\de_{c}\Delta_{q_{0}}^{p})(x)=2p\Delta_{q_{0}}^{p-1}(x)(x-\Re(q_{0}))$, which clearly vanishes for $x=q_{0}$.
Assume now that $k>0$ and $p>k+1$. By Proposition~\ref{main} and Corollary~\ref{cor1}, it holds
\begin{align*}
((\de_{c}(\de_{s}\de_{c})^{k})\Delta_{q_{0}}^{p})(x)&=(\de_{c}\de_{s})^{k}(2p\Delta_{q_{0}}^{p-1}(x)(x-\Re(q_{0})))\\
&=2p\left\{((\de_{c}\de_{s})^{k}\Delta_{q_{0}}^{p-1})(x)\cdot v_{s}(x-\Re(q_{0}))+\frac{1}{2}\left[\de_{s}(\de_{c}\de_{s})^{k-1}\Delta_{q_{0}}^{p-1}(x)+(\de_{c}(\de_{s}\de_{c})^{k-1}\Delta_{q_{0}}^{p-1})(x)\right]
\right\}
\end{align*}
Now, for $x=q_{0}$ we have that $(v_{s}(x-\Re(q_{0})))_{|_{x=q_{0}}}=0$ and, by Lemma~\ref{lemma2} $(\de_{s}(\de_{c}\de_{s})^{k-1}\Delta_{q_{0}}^{p-1})(q_{0})=0$. Hence, the only piece that survives is $p(\de_{c}(\de_{s}\de_{c})^{k-1}\Delta_{q_{0}}^{p-1})(q_{0})$. Iterating the same computation and argument $k-1$ times, we obtain
\begin{align*}
((\de_{c}(\de_{s}\de_c)^{k})\Delta_{q_{0}}^{p})(q_{0})&=p\cdot (p-1)\cdot\ldots\cdot(p-(k-1))(\de_{c}\Delta_{q_{0}}^{p-k})(q_{0})\\
&=p\cdot (p-1)\cdot\ldots\cdot(p-(k-1))\Big[2(p-k)\Delta_{q_{0}}^{p-k-1}(q_{0})(q_{0}-\Re(q_{0}))\Big](q_{0})=0
\end{align*}
Assume now that $k=p-1$. For $k=1$ we get
$$
\de_{c}\Delta_{q_{0}}=2(x-\Re(q_{0})),
$$
which, for $x=q_{0}$ is equal to $2\Im(q_{0})$. 
Now, for a general $k>1$, repeating the  previous computation, we get
$$
((\de_{c}(\de_{s}\de_{c})^{k-1})\Delta_{q_{0}}^{k})(q_{0})=k\cdot (k-1)\cdot\ldots\cdot1(\de_{c}\Delta_{q_{0}})(q_{0})=2\Im(q_{0})k!,$$
which gives the thesis.
%
%
%
%
\end{proof}

We now pass to study $\Delta_{q_{0}}^{k}(x)(x-q_{0})$. First of all we saw in Corollary~\ref{cor1} that, when $k=0$, $\de_{s}(x-q_{0})\equiv 1$.
We will prove a suitable generalization in the following result.

\begin{theorem}\label{lemma7}
Let $q_{0}\in\HH$ and $p\in\mathbb{N}\setminus\{0\}$. Then, for any $0\leq k\leq p$ we have
$$
(\de_{s}(\de_{c}\de_{s})^{k}(\Delta_{q_{0}}^{p}(x)(x-q_{0})))_{|_{x=q_{0}}}=
\begin{cases}
k!,\qquad\mbox{ if }k=p,\\
0,\qquad\mbox{ otherwise.}
\end{cases}
$$
\end{theorem}
\begin{proof}
The result is trivially true when $k=0$ as the spherical value and derivative of $\Delta_{q_{0}}^{p}$ vanish at $q_{0}$.
Let now $k>0$ and $p>k$. By Proposition~\ref{main}, we have
\begin{align*}
(\de_{s}(\de_{c}\de_{s})^{k})(\Delta_{q_{0}}^{p}(x)(x-q_{0}))&=\de_{s}\Big\{((\de_{c}\de_{s})^{k}\Delta_{q_{0}}^{p}(x)) v_{s}(x-q_{0})\\
&\quad\quad\quad+\frac{1}{2}\left[\binom{k-1}{0}\de_{s}(\de_{c}\de_{s})^{k-1}\Delta_{q_{0}}^{p}(x)+\binom{k-1}{k-1}\de_{c}(\de_{s}\de_{c})^{k-1}\Delta_{q_{0}}^{p}(x)\right]\Big\}.
\end{align*}
Using formulas~\eqref{basicslice}, we conclude that,
$$
(\de_{s}(\de_{c}\de_{s})^{k})(\Delta_{q_{0}}^{p}(x)(x-q_{0}))=(\de_{s}(\de_{c}\de_{s})^{k}\Delta_{q_{0}}^{p})\cdot v_{s}(x-q_{0})+\frac{1}{2}\left[(\de_{s}\de_{c})^{k}\Delta_{q_{0}}^{p}\right],
$$
and thanks to Lemmas~\ref{lemma2} and~\ref{lemma5} we get the thesis.

If instead, $k=p$, Lemmas~\ref{lemma2} and~\ref{lemma5} imply that
$$
(\de_{s}(\de_{c}\de_{s})^{k})(\Delta_{q_{0}}^{k}(x)(x-q_{0}))_{|_{x=q_{0}}}=\frac{1}{2}\left[((\de_{s}\de_{c})^{k}\Delta_{q_{0}}^{k})\right](q_{0})=\frac{1}{2}\cdot2k!=k!
$$
obtaining the last part of the thesis.
\end{proof}

%

The following is the last of our technical lemmas.
\begin{lemma}\label{lemma8}
Let $q_{0}\in\HH$ and $p\in\mathbb{N}\setminus\{0\}$. Then, for any $0\leq k\leq p$ we have
$$
((\de_{c}\de_{s})^{k}(\Delta_{q_{0}}^{p}(x)(x-q_{0})))_{|_{x=q_{0}}}=0.
$$
\end{lemma}
\begin{proof}
For $k=0$ the result is trivial as $(\Delta_{q_{0}}^{p}(x)(x-q_{0}))(q_{0})=0$ for any $p\in\mathbb{N}$. 

If $1=k<p$, we get
$$
(\de_{c}\de_{s})\Delta_{q_{0}}^{p}(x)(x-q_{0})=((\de_{c}\de_{s})\Delta_{q_{0}}^{p}) v_{s}(x-q_{0})+\frac{1}{2}\left[\de_{s}\Delta_{q_{0}}^{p}+\de_{c}\Delta_{q_{0}}^{p}\right],
$$
and as $p$ is at least equal to $2$, then everything vanishes at $x=q_{0}$ and we obtain that $((\de_{c}\de_{s})(\Delta_{q_{0}}^{p}(x-q_{0})))(q_{0})=0$.
If $1<k<p$, then
\begin{align*}
((\de_{c}\de_{s})^{k})(\Delta_{q_{0}}^{p}(x)(x-q_{0}))&=((\de_{c}\de_{s})^{k}\Delta_{q_{0}}^{p}(x))\cdot v_{s}(x-q_{0})\\
&\quad\quad\quad+\frac{1}{2}\left[\binom{k-1}{0}\de_{s}(\de_{c}\de_{s})^{k-1}\Delta_{q_{0}}^{p}(x)+\binom{k-1}{k-1}\de_{c}(\de_{s}\de_{c})^{k-1}\Delta_{q_{0}}^{p}(x)\right].
\end{align*}
Now, as we are assuming that $k<p$, by Lemmas~\ref{lemma2}, ~\ref{lemma3} and Theorem~\ref{lemma6}, the previous
expression vanishes at $x=q_{0}$.

If now $k=p$, repeating the previous computation we get
$$
((\de_{c}\de_{s})^{k})(\Delta_{q_{0}}^{k}(x)(x-q_{0}))=((\de_{c}\de_{s})^{k}\Delta_{q_{0}}^{k}(x))\cdot v_{s}(x-q_{0})+\frac{1}{2}\left[\de_{s}(\de_{c}\de_{s})^{k-1}\Delta_{q_{0}}^{k}(x)+\de_{c}(\de_{s}\de_{c})^{k-1}\Delta_{q_{0}}^{k}(x)\right].
$$
Now using Theorem~\ref{lemma6}, Lemma~\ref{lemma2} and Lemma~\ref{lemma3} we get that
$$
(((\de_{c}\de_{s})^{k})(\Delta_{q_{0}}^{k}(x)(x-q_{0})))_{|_{x=q_{0}}}=k!\cdot (-\Im(q_{0}))+\frac{1}{2}\left[2k!\cdot \Im(q_{0})\right]=0.
$$
\end{proof}

\subsection{A comprehensive table}\label{tabel1}
For the convenience of the reader we collect all the non-vanishing results obtained  in the previous computations in a handy table
\begin{table}[h!]
\centering
\begin{tabular}{ |l||c|c|  }
 \hline
 \multicolumn{3}{|c|}{A summary of the previous results} \\
 \hline
 Result& Index values &Result reference\\
 \hline
$(\de_{s}\de_{c})^{k}\Delta_{q_{0}}^{k}=2k!$ & $k\geq 1$    & Lemma~\ref{lemma5} \\
$(\de_{c}\de_{s})^{k}\Delta_{q_{0}}^{k}=k!$ & $k\geq 1$    & Theorem~\ref{lemma6} \\
$\de_{s}(\de_{c}\de_{s})^{k}(\Delta_{q_{0}}^{k}(x)(x-q_{0}))=k!$ & $k\geq 0$    & Theorem~\ref{lemma7} \\
$(\de_{c}(\de_{s}\de_{c})^{k-1}(\Delta_{q_{0}}^{k}))_{|_{x=q_{0}}}=2\Im(q_{0})k!$ & $k\geq 1$    & Lemma~\ref{lemma3} \\
 \hline
\end{tabular}
\end{table}
\section{Coefficients of the spherical expansion}\label{coefficients}
Thanks to the results of the previous section, we are able to give a differential representation of the coefficients of the spherical expansion of a slice regular function.
First of all recall that if $f\in\So(\Omega_{D})$ and $q_{0}\in\Omega_{D}\setminus\R$, then there exists a symmetric domain, namely a Cassini ball $U(q_{0},r)\subset\Omega_{D}$, such that,
for any $x\in U(q_{0},r)$ we have
\begin{equation}\label{sphericalexpansion}
f(x)=s_{0}+(x-q_{0})s_{1}+\Delta_{q_{0}}(x)s_{2}+\Delta_{q_{0}}(x)(x-q_{0})s_{3}+\dots+\Delta_{q_{0}}^{k}(x)s_{2k}+(\Delta_{q_{0}}^{k}(x)(x-q_{0}))s_{2k+1}+\dots.
\end{equation}
Now, it is well known (see~\cite{altavilladiff,G-P-series,stoppatosph}) that $s_{0}=f(q_{0})$, $s_{1}=(\de_{s}f)(q_{0})$ and that $s_{2}=((\de_{c}\de_{s})f)(q_{0})$. We are able to describe in the same way all the other coefficients.
\begin{theorem}\label{sphcoeff}
Let $f$ be a slice regular function defined on a symmetric domain $\Omega_{D}$ and let $q_{0}\in\Omega_{D}\setminus\R$. 
Then, we have that the coefficients of the spherical expansion~\eqref{sphericalexpansion} are given as follows
$$
\begin{cases}
s_{2k}=\frac{1}{k!}((\de_{c}\de_{s})^{k}f)(q_{0}),\\
s_{2k+1}=\frac{1}{k!}(\de_{s}(\de_{c}\de_{s})^{k}f)(q_{0}).
\end{cases}
$$
\end{theorem}
\begin{proof}
As the result is trivially true for $k=0$, we start from $k=1$. Now, even for $k=1$ we already know that $s_{2}=((\de_{c}\de_{s})f)(q_{0})$. We then work on $s_{3}$.
If we apply the differential operator $\de_{s}(\de_{c}\de_{s})$ to $f$ as represented in Formula~\eqref{sphericalexpansion}, we obtain
\begin{align*}
(\de_{s}(\de_{c}\de_{s})f)(x)=&(\de_{s}(\de_{c}\de_{s}\Delta_{q_{0}}(x)(x-q_{0}))s_{3}+(\de_{s}(\de_{c}\de_{s})\Delta_{q_{0}}^{2}(x))s_{4}\\
&+\dots+(\de_{s}(\de_{c}\de_{s})\Delta_{q_{0}}^{k}(x))s_{2k}+(\de_{s}(\de_{c}\de_{s})(\Delta_{q_{0}}^{k}(x)(x-q_{0})))s_{2k+1}+\dots.
\end{align*}
But now, thanks to Theorem~\ref{lemma7} and to Lemma~\ref{lemma2} we have
$$
(\de_{s}(\de_{c}\de_{s})f)(q_{0})=s_{3}.
$$

Assume now that $k>1$.  We start from the even coefficients. If we apply the differential operator $(\de_{c}\de_{s})^{k}$ to $f$ as represented in Formula~\eqref{sphericalexpansion}, we obtain
\begin{align*}
((\de_{c}\de_{s})^{k}f)(x)=&((\de_{c}\de_{s})^{k}\Delta_{q_{0}}^{k}(x))s_{2k}+((\de_{c}\de_{s})^{k}\Delta_{q_{0}}^{k}(x)(x-q_{0}))s_{2k+1}\\
&+\dots+((\de_{c}\de_{s})^{k}\Delta_{q_{0}}^{p}(x))s_{2p}+((\de_{c}\de_{s})^{k}(\Delta_{q_{0}}^{p}(x)(x-q_{0})))s_{2p+1}+\dots.
\end{align*}
But now, thanks to Theorem~\ref{lemma6} and to Lemma~\ref{lemma8} we have
$$
((\de_{c}\de_{s})^{k}f)(q_{0})=k!s_{2k}.
$$
Analogously we have 
\begin{align*}
(\de_{s}(\de_{c}\de_{s})^{k}f)(x)=&(\de_{s}(\de_{c}\de_{s})^{k}\Delta_{q_{0}}^{k}(x)(x-q_{0}))s_{2k+1}+(\de_{s}(\de_{c}\de_{s})^{k}\Delta_{q_{0}}^{k+1}(x))s_{2k+2}\\
&+\dots+(\de_{s}(\de_{c}\de_{s})^{k}\Delta_{q_{0}}^{p}(x))s_{2p}+(\de_{s}(\de_{c}\de_{s})^{k}\Delta_{q_{0}}^{p}(x)(x-q_{0}))s_{2p+1}+\dots,
\end{align*}
and thanks to Theorem~\ref{lemma7} and to Lemma~\ref{lemma2} we have
$$
(\de_{s}(\de_{c}\de_{s})^{k}f)(q_{0})=k!s_{2k+1},
$$
concluding the proof.
\end{proof}

We give a first corollary about slice preserving and one-slice preserving functions.
\begin{corollary}
Let $f:\Omega_{D}\to\HH$ be a slice regular function, $q_{0}\in\Omega_{D}\setminus\R$ and let $s_{n}(q_{0})$ be the spherical coefficients of $f$ at $q_{0}$.
Then the following facts hold true:
\begin{enumerate}
\item the function $f$ is slice preserving if and only if, for any $k\in\mathbb{N}$, $s_{2k+1}(q_{0})\in\R$ and, if $q_{0}\in\C_{I}$, $s_{2k}(q_{0})\in\C_{I}$.
\item the function $f$ is $\C_{J}$-preserving if and only if $s_{2k+1}(q_{0})\in\C_{J}$ and, for any $q_{0}\in\C_{J}$,
$s_{2k}(q_{0})\in\C_{J}$.
\end{enumerate}
\end{corollary}
\begin{proof}
The necessity follows from the expression of the spherical coefficients given in Theorem~\ref{sphcoeff} while the sufficiency follows directly just by looking at $s_{0}$ and $s_{1}$.
\end{proof}

In view of the next corollary we recall some basic fact about zeros of slice regular functions. Let $f:\Omega_{D}\to \HH$ be a slice regular function defined on a symmetric
domain and $q_{0}\in\Omega_{D}\setminus\R$. Let $m,n\in\mathbb{N}$ and $q_{1},\dots,q_{n}\in\SF_{q_{0}}$ (with $q_{i}\neq\bar q_{i+1}$ for all $i=1,\dots, n-1$) be such that
$$
f(x)=\Delta_{q_{0}}^{m}(x)(x-q_{1})\cdot (x-q_{2})\cdot \dots\cdot (x-q_{n})\cdot g(x),
$$
for some slice regular function $g:\Omega_{D}\to\HH$ which does not have any zeros in $\SF_{q_{0}}$. We then say that $2m$ is the spherical multiplicity
of $\SF_{q_{0}}$ and that $n$ is the isolated multiplicity of $q_{1}$.
As it was pointed out in~\cite{stoppatosph} and in~\cite[Remark 8.11]{G-S-St}, if $f$ is a slice regular function on a symmetric domain $\Omega_{D}$
and $q_{0}\in\Omega_{D}\setminus\R$, then if $s_{2n}$ or $s_{2n+1}$ is the first non vanishing coefficient in the spherical expansion~\eqref{sphericalexpansion},
then $2n$ is the spherical multiplicity of $f$ at $\SF_{q_{0}}$. Moreover, $q_{0}$ has a positive isolated multiplicity if and only if $s_{2n}=0$. These
information can be translated in new formulas thanks to our results.
\begin{corollary}
Let $f$ be a slice regular function defined on a symmetric domain $\Omega_{D}$ and let $q_{0}\in\Omega_{D}\setminus\R$. 
Then, if for any $k=0,\dots, n-1$ we have that $((\de_{c}\de_{s})^{k}f)(q_{0})=(\de_{s}(\de_{c}\de_{s})^{k}f)(q_{0})=0$ and 
$((\de_{c}\de_{s})^{n}f)(q_{0})\neq 0$ or $(\de_{s}(\de_{c}\de_{s})^{n}f)(q_{0})\neq0$
then $2n$ is the spherical multiplicity of $f$ at $\SF_{q_{0}}$. Moreover, $q_{0}$ has a positive isolated multiplicity if and only if $((\de_{c}\de_{s})^{n}f)(q_{0})=0$. 
\end{corollary}

Having now these new nice representations of the spherical coefficients of a slice regular functions, we are able to effectively compute examples.

\begin{example}\label{expel}
We give an explicit example for the previous corollary. Consider the quaternionic polynomial 
$$
P(x)=x^{3}-x^{2}(i+j+k)+x(i-j+k)+1=(x-i)\cdot (x-j)\cdot (x-k).
$$
We want to compute its spherical coefficients in general (i.e. at a non-real point) and at $x_{0}=i$.
If $s_{n}(x)$ denote the $n$-th spherical coefficient at $x=\alpha+I\beta\in\HH\setminus\R$, then we have
\begin{align*}
& &s_{0}(i)=0,\\
 &s_{1}(x)=(\de_{s}P)(x)=3\alpha^{2}-\beta^{2}-2\alpha(i+j+k)+i-j+k, &s_{1}(i)=-1+i-j+k,\\
 &s_{2}(x)=(\de_{c}\de_{s}P)(x)=3\alpha-(i+j+k)+I\beta, &s_{2}(i)=-j-k,\\
 &s_{3}(x)=(\de_{s}\de_{c}\de_{s}P)(x)=1 &s_{3}(i)=1,
\end{align*}
and so, for any $n>3$, we have $s_{n}(x)=0$.
Therefore, near $x_{0}=i$, we have
$$
P(x)=(x-i)(-1+i-j+k)+\Delta_{i}(x)(-j-k)+\Delta_{i}(x)(x-i).
$$
\end{example}

We now deal with basic idempotent functions~\cite{A-dFLAA,A-S}.
\begin{definition}
We define the slice regular function $\mathcal{J}:\mathbb{H}\setminus\R\to \HH$ as $\mathcal{J}(x)=\frac{\Im(x)}{|\Im(x)|}$, for all $x\in\HH\setminus\R$.
Sometimes, if $x=\alpha+I\beta$, we will also write $\mathcal{J}(x)=I$.

Let now $J\in\SF$ and $x=\alpha+I\beta\in\HH\setminus\R$, with $\beta>0$. We define $\ell^{+,J}:\HH\setminus\R\to\HH$, $\ell^{-,J}:\HH\setminus\R\to\HH$ as
$$
\ell^{+,J}(x)=\frac{1-\mathcal{J}(x)J}{2},\qquad\ell^{-,J}(x)=\frac{1+\mathcal{J}(x)J}{2}
$$
\end{definition}

As explained in~\cite{A-dFLAA,A-S} these functions $\ell^{\pm,J}$ are idempotents for the slice product and any slice regular function $f$ defined on a symmetric domain without real points can be written
as $f=f_{+}\cdot \ell^{+,J}+f_{-}\cdot \ell^{-,J}$, where $f_{+}$ and $f_{-}$ are slice regular functions defined on the same domain of $f$ (this is called Peirce decomposition). 
The important role of idempotent function was also exploited in~\cite{M-S} in order to characterize some relevant space of measurable regular functions.
Before going into the details, recall that, for any $N\in\mathbb{N}$, the notation $N!!$ stands for the double factorial, i.e. 
$$
N!!=\prod_{k=0}^{\lceil \frac{n}{2}\rceil-1}(N-2k)=N(N-2)(N-4)\cdots
$$

\begin{corollary}\label{idempotents}
Let $J\in\SF$ be any imaginary unit and $x=\alpha+I\beta\in\HH\setminus\R$, with $\beta>0$. Then the spherical coefficients of $\ell^{+,J}$ at $x$ are the following
$$
s_{0}^{+}(\alpha+I\beta)=\frac{1-IJ}{2},\quad\begin{cases}
s_{2k}^{+}(\alpha+I\beta)=-\frac{1}{k!}\frac{(2k-1)!!}{2^{k+1}}\frac{1}{\beta^{2k}}IJ,\\
s_{2k+1}^{+}(\alpha+I\beta)=-\frac{1}{k!}\frac{(2k-1)!!}{2^{k+1}}\frac{1}{\beta^{2k+1}}J.
\end{cases}
$$
Analogously, the spherical coefficients of $\ell^{-,J}$ at $x$ are the following
$$
s_{0}^{-}(\alpha+I\beta)=\frac{1+IJ}{2},\quad\begin{cases}
s_{2k}^{-}(\alpha+I\beta)=\frac{1}{k!}\frac{(2k-1)!!}{2^{k+1}}\frac{1}{\beta^{2k}}IJ,\\
s_{2k+1}^{-}(\alpha+I\beta)=\frac{1}{k!}\frac{(2k-1)!!}{2^{k+1}}\frac{1}{\beta^{2k+1}}J.
\end{cases}
$$
\end{corollary}
\begin{proof}
We will prove the result for $\ell^{+,J}$ as the argument for $\ell^{-,J}$ is completely analogous.
For $s_{0}^{+}$ there is nothing to prove. We have 
$$s_{1}^{+}(\alpha+I\beta)=(\de_{s}\ell^{+,J})(\alpha+I\beta)=-\frac{1}{2}\frac{J}{\beta},$$
and so
$$
s_{2}^{+}(\alpha+I\beta)=(\de_{c}\de_{s}\ell^{+,J})(\alpha+I\beta)=\de_{c}\left(-\frac{1}{2}\frac{J}{\beta}\right)=\frac{1}{2}\left(\frac{\de}{\de\alpha}-I\frac{\de}{\de\beta}\right)\left(-\frac{1}{2}\frac{J}{\beta}\right)=
-\frac{1}{2^{2}}\frac{1}{\beta^{2}}IJ.
$$
We now compute $s_{3}$
$$
s_{3}^{+}(\alpha+I\beta)=(\de_{s}\de_{c}\de_{s}\ell^{+,J})(\alpha+I\beta)=\de_{s}\left(-\frac{1}{2^{2}}\frac{1}{\beta^{2}}IJ\right)(\alpha+I\beta)=-\frac{1}{2^{2}}\frac{1}{\beta^{3}}J.
$$
Let now $k>1$ be fixed and assume that the result is true for all integers lower than $k$.
From Theorem~\ref{sphcoeff} we know that
$$
\begin{cases}
s_{2k}^{+}(\alpha+I\beta)=\frac{1}{k!}((\de_{c}\de_{s})^{k}\ell^{+,J})(\alpha+I\beta),\\
s_{2k+1}^{+}(\alpha+I\beta)=\frac{1}{k!}(\de_{s}(\de_{c}\de_{s})^{k}\ell^{+,J})(\alpha+I\beta),
\end{cases}
$$
therefore 
\begin{align*}
k!s_{2k}^{+}(\alpha+I\beta)=&((\de_{c}\de_{s})^{k}\ell^{+,J})(\alpha+I\beta)=(\de_{c}(\de_{s}(\de_{c}\de_{s})^{k-1}\ell^{+,J}))(\alpha+I\beta)\\
&=\de_{c}\left((k-1)!s_{2(k-1)+1}(\alpha+I\beta)\right)\\
&=\de_{c}\left(-\frac{(2(k-1)-1)!!}{2^{(k-1)+1}}\frac{1}{\beta^{2(k-1)+1)}}J\right)\\
&=\frac{1}{2}\left(-I\frac{\de}{\de\beta}\right)\left(-\frac{(2(k-1)-1)!!}{2^{(k-1)+1}}\frac{1}{\beta^{2k-1}}J\right)\\
&=\frac{1}{2}\left(\frac{(2(k-1)-1)!!}{2^{(k-1)+1}}\left(-\frac{2k-1}{\beta^{2k}}\right)IJ\right)\\
&=-\frac{(2k-1)!!}{2^{k+1}}\frac{1}{\beta^{2k}}IJ.
\end{align*}
Analogously, we have
$$
k!s_{2k+1}^{+}(\alpha+I\beta)=(\de_{s}(\de_{c}\de_{s})^{k}\ell^{+,J})(\alpha+I\beta)=\de_{s}\left(-\frac{(2k-1)!!}{2^{k+1}}\frac{1}{\beta^{2k}}IJ\right)=-\frac{(2k-1)!!}{2^{k+1}}\frac{1}{\beta^{2k+1}}J.
$$
\end{proof}

\begin{remark}
In~\cite[Example 5.6]{G-P-series} the authors gave a description of spherical coefficients $\tilde s_{n}$ for the function $2\ell^{+,J}$ at $x=J$. They found
$$
\tilde s_{0}=2,\qquad\begin{cases}
\tilde s_{2k+1}=4^{-k}\binom{2k}{k}(-J),\quad k\geq0,\\
\tilde s_{2k}=4^{-k}\binom{2k}{k},\quad k\geq1.
\end{cases}
$$
Our result is consistent as, at $x=J$ (and $\beta=1$), we have
$$
2s_{0}^{+}(J)=2\cdot\frac{1-J\cdot J}{2}=2=\tilde s_{0},\quad\begin{cases}
2s_{2k+1}^{+}(J)=2\cdot(-\frac{1}{k!}\frac{(2k-1)!!}{2^{k+1}}J)=-\frac{1}{k!}\frac{(2k-1)!!}{2^{k}}J,\\
2s_{2k}^{+}(J)=2\cdot(-\frac{1}{k!}\frac{(2k-1)!!}{2^{k+1}}JJ)=\frac{1}{k!}\frac{(2k-1)!!}{2^{k}}.
\end{cases}
$$
Now, thanks to standard properties of the double factorial, we have that 
$$(2k-1)!!=\frac{(2k)!}{2^{k}k!},$$
and hence
$$
\begin{cases}
2s_{2k+1}^{+}(J)=-\frac{1}{k!}\frac{(2k-1)!!}{2^{k}}J=-\frac{1}{k!}\frac{1}{2^{k}}\frac{(2k)!}{2^{k}k!}J=4^{-k}\binom{2k}{k}(-J)=\tilde s_{2k+1},\\
2s_{2k}^{+}(J)=\frac{1}{k!}\frac{(2k-1)!!}{2^{k}}=\frac{1}{k!}\frac{1}{2^{k}}\frac{(2k)!}{2^{k}k!}=4^{-k}\binom{2k}{k}=\tilde s_{2k}.
\end{cases}
$$
\end{remark}

\begin{remark}
Notice that, even though the functions $\ell^{\pm,J}$ are  slice constant (see~\cite{A-CVEE}), their spherical expansion is always infinite. The same is true for
slice polynomial functions, i.e. slice regular functions of the form $P=P_{+}\cdot \ell^{+,J}+P_{-}\cdot \ell^{-,J}$, where $P_{+}$ and $P_{-}$ are regular polynomials.
Using the linearity of the operators $(\de_{c}\de_{s})^{k}$ and $\de_{s}(\de_{c}\de_{s})^{k}$ and the formula in Proposition~\ref{main} one can compute all the coefficients
and, since those of $\ell^{\pm,J}$ are infinite, then also those of $P$ are, in general, infinite.
However, with some more information, it is possible to simplify a lot all the computation as the following examples show.
\end{remark}

\begin{example}\label{slicepol}
Let $P_{2n}$ be a quaternionic regular polynomial of degree $2n$ and $x=\alpha+I\beta\in\HH\setminus\R$. Then the spherical coefficients of the slice polynomial function $f=\ell^{+,J}\cdot P_{2n}:\HH\setminus\R\to\HH$ can be computed using Proposition~\ref{main}, Corollary~\ref{idempotents} and the fact that $\partial_{c}\ell^{+,J}\equiv 0$, as follows.
\begin{align*}
s_{2k}(x)=\frac{1}{k!}((\de_{c}\de_{s})^{k}f)(x)&=\frac{1}{k!}\Bigg\{\left(-\frac{(2k-1)!!}{2^{k+1}}\frac{1}{\beta^{2k}}IJ\right)\cdot (v_{s}P_{2n})(x)+\frac{1}{2}((\de_{c}\de_{s})^{k}P_{2n})(x)\\
&\quad+\frac{1}{2}\left[\sum_{h=0}^{k-1}\binom{k-1}{h}\left(-\frac{(2(k-(h+1))-1)!!}{2^{(k-(h+1))+1}}\frac{1}{\beta^{2(k-(h+1))+1}}J\right)\cdot (\de_{c}(\de_{s}\de_{c})^{h}P_{2n})(x)\right.\\
&\left.\quad\quad+\sum_{h=1}^{k-1}\binom{k-1}{h}\left(-\frac{(2(k-h)-1)!!}{2^{(k-h)+1}}\frac{1}{\beta^{2(k-h)}}IJ\right)\cdot ((\de_{s}\de_{c})^{h}P_{2n})(x)\right]\Bigg\}
\end{align*}
and if $k>n$ it further simplifies as
\begin{align*}
s_{2k}(x)&=\frac{1}{k!}\Bigg\{\left(-\frac{(2k-1)!!}{2^{k+1}}\frac{1}{\beta^{2k}}IJ\right)\cdot (v_{s}P_{2n})(x)\\
&\quad+\frac{1}{2}\left[\sum_{h=0}^{n-1}\binom{k-1}{h}\left(-\frac{(2(k-(h+1))-1)!!}{2^{(k-(h+1))+1}}\frac{1}{\beta^{2(k-(h+1))+1}}J\right)\cdot (\de_{c}(\de_{s}\de_{c})^{h}P_{2n})(x)\right.\\
&\left.\quad\quad+\sum_{h=1}^{n-1}\binom{k-1}{h}\left(-\frac{(2(k-h)-1)!!}{2^{(k-h)+1}}\frac{1}{\beta^{2(k-h)}}IJ\right)\cdot ((\de_{s}\de_{c})^{h}P_{2n})(x)\right]\Bigg\}
\end{align*}

Analogous considerations hold for $s_{2k+1}$, for $\ell^{-,J}$ and for an odd-degree polynomial $P$.
\end{example}

 Instead of expanding the previous example in all its cases, we analyze a specific case, that was extensively studied in~\cite{A-S} in relation to a geometric interpretation of slice regularity.
 
 \begin{example}
Let $R:\HH\setminus\R\to\HH$ be defined as $R(x)=-x^{2}\cdot \ell^{+,i}+x\cdot \ell^{-,i}$.  We will write the spherical coefficient in general and at $x_{0}=\frac{1}{2}+\frac{\sqrt{3}}{2}i$.
Let $x=\alpha+I\beta$, by direct inspection we have
$$
\de_{s}(-x^{2})=-2\alpha,\quad v_{s}(x^{2})=-(\alpha^{2}-\beta^{2}),\quad (\de_{c}\de_{s})(-x^{2})=-1,(\de_{s}\de_{c})(-x^{2})=-2,
$$
and so,
%
\begin{align*}
& &s_{0}(x_{0})=x_{0}^{c}=\frac{1}{2}-\frac{\sqrt{3}}{3}i,\\
 &s_{1}(x)=-\alpha+\frac{1}{2}+\frac{\alpha^{2}-\beta^{2}+\alpha}{2\beta}i, &s_{1}(x_{0})=0,\\
 &s_{2}(x)=-\frac{1}{2}-x\left(-\frac{i}{2\beta}\right)-(\alpha^{2}-\beta^{2})\left(-\frac{Ii}{4\beta^{2}}\right)+\frac{1}{2}\left(-\frac{i}{2\beta}\right)+\alpha\left(\frac{Ii}{4\beta^{2}}\right), &s_{2}(x_{0})=-1+\frac{\sqrt{3}}{3}i,\\
 &s_{3}(x)=-\frac{i}{2\beta}-(\alpha^{2}-\beta^{2})\left(-\frac{i}{4\beta^{3}}\right)+\alpha\left(\frac{i}{4\beta^{3}}\right), &s_{3}(x_{0})=\frac{\sqrt{3}}{3}i.
\end{align*}
%
Assume now $k\geq 2$, then, thanks to Proposition~\ref{main}, we have that
\begin{align*}
((\de_{c}\de_{s})^{k}R)(x)&=v_{s}(-x^{2})((\de_{c}\de_{s})^{k}\ell^{+,i})(x)+v_{s}(x)((\de_{c}\de_{s})^{k}\ell^{-,i})(x)\\
&\quad+\frac{1}{2}\left\{\de_{c}(-x^{2})(\de_{s}(\de_{c}\de_{s})^{k-1}\ell^{+,i})(x)+(k-1)((\de_{s}\de_{c})(-x^{2}))((\de_{c}\de_{s})^{k-1}\ell^{+,i})(x)\right.\\
&\quad\quad\quad\left.+\de_{c}(x)(\de_{s}(\de_{c}\de_{s})^{k-1}\ell^{-,i})
\right\}\\
&=(\alpha^{2}-\beta^{2})\left[\frac{(2k-1)!!}{2^{k+1}\beta^{2k}}Ii\right]+\alpha\left[\frac{(2k-1)!!}{2^{k+1}\beta^{2k}}Ii\right]\\
&\quad+\frac{1}{2}\left\{2x\left(\frac{(2(k-1)-1)!!}{2^{k}\beta^{2(k-1)+1}}i\right)+2(k-1)\left(\frac{(2(k-1)-1)!!}{2^{k}\beta^{2(k-1)}}Ii\right)+\frac{(2(k-1)-1)!!}{2^{k}\beta^{2(k-1)+1}}i\right\}\\
&=\frac{(2(k-1)-1)!!}{2^{k}\beta^{2(k-1)}}\left[\frac{1}{\beta}(\alpha+\frac{1}{2})+\left(k+\frac{2k-1}{2\beta^{2}}(\alpha^{2}-\beta^{2}+\alpha)\right)I\right]i.
\end{align*}
Therefore, 
$$
((\de_{c}\de_{s})^{k}R)(x_{0})=\frac{2^{k-2}}{3^{k-1}}(2(k-1)-1)!!\left[-k+\frac{2\sqrt{3}}{3}i\right]=\frac{2^{k-2}}{3^{k-1}}(2k-3)!!\left[-k+\frac{2\sqrt{3}}{3}i\right],
$$
and hence
$$
s_{2k}(x_{0})=\frac{2^{k-2}}{k!3^{k-1}}(2k-3)!!\left[-k+\frac{2\sqrt{3}}{3}i\right].
$$
Assuming again that $k\geq2$, for the odd coefficients, we have,

\begin{align*}
(\de_{s}(\de_{c}\de_{s})^{k}R)(x)&=v_{s}(-x^{2})(\de_{s}(\de_{c}\de_{s})^{k}\ell^{+,i})(x)+v_{s}(x)(\de_{s}(\de_{c}\de_{s})^{k}\ell^{-,i})(x)\\
&\quad+\frac{1}{2}\left\{(\de_{s}\de_{c})(-x^{2})(\de_{s}(\de_{c}\de_{s})^{k-1}\ell^{+,i})(x)+(k-1)((\de_{s}\de_{c})(-x^{2}))(\de_{s}(\de_{c}\de_{s})^{k-1}\ell^{+,i})(x)\right\}\\
&=v_{s}(-x^{2})(\de_{s}(\de_{c}\de_{s})^{k}\ell^{+,i})(x)+v_{s}(x)(\de_{s}(\de_{c}\de_{s})^{k}\ell^{-,i})(x)+\frac{k}{2}(\de_{s}\de_{c})(-x^{2})(\de_{s}(\de_{c}\de_{s})^{k-1}\ell^{+,i})(x)\\
&=(\alpha^{2}-\beta^{2})\left(\frac{(2k-1)!!}{2^{k+1}\beta^{2k+1}}i\right)+\alpha\left(\frac{(2k-1)!!}{2^{k+1}\beta^{2k+1}}i\right)+k\left(\left(\frac{(2(k-1)-1)!!}{2^{k}\beta^{2k-1}}i\right)\right)\\
&=\frac{(2k-3)!!}{2^{k}\beta^{2k-1}}\left[k+\frac{2k-1}{2\beta^{2}}(\alpha^{2}-\beta^{2}+\alpha)\right]i.
\end{align*}
Therefore, 
$$
((\de_{c}\de_{s})^{k}R)(x_{0})=\frac{2^{k-1}}{3^{k-1}\sqrt{3}}(2k-3)!!\left[ki\right]=\frac{2^{k-1}}{3^{k}}(2k-3)!!\left[\sqrt{3}ki\right],
$$
and hence
$$
s_{2k+1}(x_{0})=\frac{2^{k-1}}{k!3^{k}}(2k-3)!!\left[\sqrt{3}ki\right].
$$

\end{example}

\section{Spherical coefficients of the slice derivative}\label{slicederi}
Thanks to the results of the previous section, we are able to compute the coefficients of the spherical expansion of the slice derivative of a slice regular function.
Starting from Equation~\eqref{sphericalexpansion} we have
\begin{align}
\label{sphexpdec}\de_{c}f(x)&=s_{1}+\de_{c}(\Delta_{q_{0}}(x))s_{2}+\de_{c}(\Delta_{q_{0}}(x)(x-q_{0}))s_{3}+\dots+\de_{c}(\Delta_{q_{0}}^{k}(x))s_{2k}+\de_{c}((\Delta_{q_{0}}^{k}(x)(x-q_{0})))s_{2k+1}+\dots\\
&=s_{1}+(2(x-\Re(q_{0})))s_{2}+((x-q_{0})^{\cdot 2}+2\Delta_{q_{0}}(x))s_{3}+\dots\nonumber\\
&\quad\dots+(2k\Delta_{q_{0}}^{k-1}(x)(x-\Re(q_{0})))s_{2k}+\Delta_{q_{0}}^{k-1}(x)\left[k(x-q_{0})^{\cdot 2}+(k+1)\Delta_{q_{0}}(x)\right]s_{2k+1}+\dots\nonumber,
\end{align}
Now, if we denote by $s_{h}'$ the spherical coefficients of $\de_{c}f$, we already know, from Theorem~\ref{sphcoeff}, that
$$
\begin{cases}
s_{2k}'=\frac{1}{k!}((\de_{c}\de_{s})^{k}\de_{c}f)(q_{0}),\\
s_{2k+1}'=\frac{1}{k!}(\de_{s}(\de_{c}\de_{s})^{k}\de_{c}f)(q_{0}).
\end{cases}
$$
Therefore we have that
$$
s_{0}'=(\de_{c}f)(q_{0})=s_{1}+2(q_{0}-\Re(q_{0}))s_{2}=s_{1}+2\Im(q_{0})s_{2},
$$
and that 
$$
s_{1}'=(\de_{s}\de_{c}f)(q_{0})=2s_{2}-2\Im(q_{0})s_{3}.
$$

We can now state our result.
\begin{theorem}\label{coeffslder}
Let $f$ be a slice regular function defined on a symmetric domain $\Omega_{D}$ and let $q_{0}\in\Omega_{D}\setminus\R$. 
Then, we have that the coefficients $s_{n}'$ of the spherical expansion~\eqref{sphexpdec} are given as follows
$$
\begin{cases}
s'_{2k}=(2k+1)s_{2k+1}+(k+1)2\Im(q_{0})s_{2k+2},\\
s'_{2k+1}=(2k+2)s_{2k+2}-(k+1)2\Im(q_{0})s_{2k+3}.
\end{cases}
$$
\end{theorem}
\begin{proof}
To prove the theorem is sufficient to compare the spherical expansion of the function $\de_{c}f$ near $q_{0}$
\begin{equation}\label{eq11}
\de_{c}f=\sum_{n\geq 0}\Delta_{q_{0}}^{n}(x)[s_{2n}'+(x-q_{0})s_{2n+1}'],
\end{equation}
with that obtained by deriving the expansion of $f$, i.e.
$$
\de_{c}f=\sum_{n\geq 0}\de_{c}\left(\Delta_{q_{0}}^{n}(x)[s_{2n}+(x-q_{0})s_{2n+1}]\right).
$$
We proceed to compute this last derivative:
\begin{align}
\de_{c}f&=\sum_{n\geq 0}\de_{c}\left(\Delta_{q_{0}}^{n}[s_{2n}+(x-q_{0})s_{2n+1}]\right)\nonumber\\
&=s_{1}+\sum_{n\geq 1}\left[n\Delta_{q_{0}}^{n-1}2(x-\Re(q_{0}))(s_{2n}+(x-q_{0})s_{2n+1})+\Delta_{q_{0}}^{n}s_{2n+1}\right]\nonumber\\
&=s_{1}+\sum_{n\geq 1}\left\{\Delta_{q_{0}}^{n-1}\left[2n(x-q_{0})(s_{2n}+(x-q_{0})s_{2n+1})+2n\Im(q_{0})(s_{2n}+(x-q_{0})s_{2n+1})+\Delta_{q_{0}}s_{2n+1}\right]\right\}\nonumber\\
&=s_{1}+\sum_{n\geq 1}\left\{\Delta_{q_{0}}^{n-1}\left[2n\Im(q_{0})s_{2n}+(x-q_{0})(2ns_{2n}-2n\Im(q_{0})s_{2n+1})+\Delta_{q_{0}}(2n+1)s_{2n+1}\right]\right\}\nonumber\\
&=\sum_{k\geq 0}\Delta_{q_{0}}^{k}\left[(2k+1)s_{2k+1}+(k+1)2\Im(q_{0})s_{2k+2}+(x-q_{0})((2k+2)s_{2k+2}-(k+1)2\Im(q_{0})s_{2k+3})\right]\label{eq22},
\end{align}
where we have used that
\begin{align*}
&2n(x-q_{0})(s_{2n}+(x-q_{0})s_{2n+1})+2n\Im(q_{0})(s_{2n}+(x-q_{0})s_{2n+1})+\Delta^{n}_{q_{0}}(x)s_{2n+1}=\\
&2n\Im(q_{0})s_{2n}+(x-q_{0})(2ns_{2n}+2n\Im(q_{0})s_{2n+1})+(x-q_{0})^{\cdot 2}2ns_{2n+1}+\Delta^{n}_{q_{0}}(x)s_{2n+1}=\\
&2n\Im(q_{0})s_{2n}+(x-q_{0})(2ns_{2n}+2n\Im(q_{0})s_{2n+1})+(x-q_{0})\cdot (x-q_{0}^{c}+q_{0}^{c}-q_{0})2ns_{2n+1}+\Delta^{n}_{q_{0}}(x)s_{2n+1}=\\
&2n\Im(q_{0})s_{2n}+(x-q_{0})(2ns_{2n}-2n\Im(q_{0})s_{2n+1})+\Delta^{n}_{q_{0}}(x)(2n+1)s_{2n+1}.
\end{align*}
Therefore, comparing Formula~\eqref{eq11} with Formula~\eqref{eq22}, by the uniqueness of the spherical expansion, we obtain the thesis.
\end{proof}

An alternative proof of the last theorem is given in Appendix~\ref{alternativeproof} where a much more computational approach is exploited.
The results in this appendix might be useful in further future applications.

\begin{remark}
Notice that, from the Formulas in Theorem~\ref{coeffslder},  if $\Im(q_{0})$ tends to zero along a slice, then we obtain the usual relations among the coefficients
of the power expansion of a function and of its first derivative. This happens as these two expansions coincide at real points.
\end{remark}


As an immediate byproduct of the last theorem, we can derive a countable family of differential equations that are satisfied by slice regular functions.
\begin{corollary}\label{infinite}
Let $f$ be a slice regular function defined on a symmetric domain $\Omega_{D}$ and let $x\in\Omega_{D}\setminus\R$. Then, for any integer $k\geq0$, $f$ satisfies the following equations
$$
\begin{cases}
(\de_{c}(\de_{s}\de_{c})^{k}f)(x)=(2k+1)(\de_{s}(\de_{c}\de_{s})^{k}f)(x)+2(k+1)\Im(x)((\de_{c}\de_{s})^{k+1}f)(x),\\
((\de_{s}\de_{c})^{k+1}f)(x)=2(k+1)v_{s}((\de_{c}\de_{s})^{k+1}f)(x)=
2(k+1)[(\de_{c}\de_{s})^{k+1}f)(x)-\Im(x)(\de_{s}(\de_{c}\de_{s})^{k+1}f)(x)].
\end{cases}
$$
\end{corollary}
\begin{proof}
It is immediate from Theorems~\ref{sphcoeff} and~\ref{coeffslder}.
\end{proof}

Notice that in~\cite{altavilladiff}, we found the first of such relations, namely
$$(\de_{c}f)(x)=(\de_{s}f)(x)+2\Im(x)((\de_{c}\de_{s})f)(x),$$

We now illustrate Theorem~\ref{coeffslder} with some examples.
\begin{example}
In Example~\ref{expel} we computed the spherical coefficients of the polynomial $P(x)=x^{3}-x^{2}(i+j+k)+x(i-j+k)+1$.
Its slice derivative is given by
$$
P'(x)=3x^{2}-2x(i+j+k)+i-j+k,
$$
therefore, if $x=\alpha+I\beta\in\HH\setminus\R$ denoting by $s'_{n}(x)$ the spherical coefficients of $P'$ at $x$, we have
$$
s'_{0}(i)=-1+i+j-k,\quad s'_{1}(x)=6\alpha-2(i+j+k),\quad s'_{2}(x)=3.
$$
Now we have that
\begin{align*}
s_{1}(x)+2\Im(x)s_{2}(x)&=3\alpha^{2}-\beta^{2}-2\alpha(i+j+k)+i-j+k+2I\beta(3\alpha-(i+j+k)+I\beta)\\
&=3\alpha^{2}-\beta^{2}-2\alpha(i+j+k)+i-j+k-2\beta^{2}-2I\beta(i+j+k)+6I\alpha\beta\\
&=3\alpha^{2}-3\beta^{2}+6I\alpha\beta-2(\alpha+I\beta)(i+j+k)+i-j+k\\
&=3x^{2}-2x(i+j+k)+i-j+k=s_{0}'(x).
\end{align*}
In particular, we have that $s_{0}'(i)=s_{1}(i)+2\Im(x)s_{2}(i)$.
Then we compute also 
\begin{align*}
2s_{2}(x)-2\Im(x)s_{3}(x)&=2(s_{2}(x)-\Im(x)s_{3})=2(3\alpha-(i+j+k)+I\beta-I\beta)\\
&=6\alpha-2(i+j+k)=s'_{1}(x).
\end{align*}
Finally
$$
3s_{3}(x)+4\Im(x)s_{4}(x)=3s_{3}(x)=3=s'_{2}(x).
$$
Of course, all these equalities hold for any $x\in\HH\setminus\R$ and so also at $i$.
\end{example}

\begin{example}
In Corollary~\ref{idempotents} we computed the spherical coefficients of the function $\ell^{+,J}$ as
$$
s_{0}^{+}(\alpha+I\beta)=\frac{1-IJ}{2},\quad\begin{cases}
s_{2k}^{+}(\alpha+I\beta)=-\frac{1}{k!}\frac{(2k-1)!!}{2^{k+1}}\frac{1}{\beta^{2k}}IJ,\\
s_{2k+1}^{+}(\alpha+I\beta)=-\frac{1}{k!}\frac{(2k-1)!!}{2^{k+1}}\frac{1}{\beta^{2k+1}}J.
\end{cases}
$$
Now, the slice derivative of $\ell^{+,J}$ is identically zero, and so, by Theorem~	\ref{coeffslder} we have
$$
\begin{cases}
(2k+1)s_{2k+1}^{+}+2(k+1)\Im(q_{0})s_{2(k+1)}^{+}=0,\\
2(k+1)s_{2(k+1)}^{+}-(k+1)2\Im(q_{0})s_{2(k+1)+1}^{+}=0.
\end{cases}
$$
Indeed, we have that 
\begin{align*}
(2k+1)s_{2k+1}^{+}+2(k+1)\Im(q_{0})s_{2(k+1)}^{+}&=(2k+1)\left(-\frac{1}{k!}\frac{(2k-1)!!}{2^{k+1}}\frac{1}{\beta^{2k+1}}J\right)\\
&\quad+2(k+1)I\beta\left(-\frac{1}{(k+1)!}\frac{(2k+1)!!}{2^{k+2}}\frac{1}{\beta^{2(k+1)}}IJ\right)\\
&=\left(-\frac{1}{k!}\frac{(2k+1)!!}{2^{k+1}}\frac{1}{\beta^{2k+1}}J\right)\\
&\quad+2(k+1)\left(\frac{1}{(k+1)!}\frac{(2k+1)!!}{2^{k+2}}\frac{1}{\beta^{2k+1}}J\right)=0,
\end{align*}
and also
\begin{align*}
2(k+1)s_{2(k+1)}^{+}-(k+1)2\Im(q_{0})s_{2(k+1)+1}^{+}&=2(k+1)\left(-\frac{1}{(k+1)!}\frac{(2k+1)!!}{2^{k+2}}\frac{1}{\beta^{2(k+1)}}IJ\right)\\
&\quad-(k+1)2I\beta \left(-\frac{1}{(k+1)!}\frac{(2k+1)!!}{2^{k+2}}\frac{1}{\beta^{2k+3}}J\right)=0.
\end{align*}
\end{example}

In the following two examples we show how, in some cases, Theorem~\ref{coeffslder} becomes an effective method to compute the spherical coefficients
of specific functions.

\begin{example}
In Example~\ref{slicepol} we computed the spherical coefficients of the slice polynomial function $R(x)=-x^{2}\ell^{+,i}+x\ell^{-,i}$.
Its slice derivative is the slice polynomial function $R'(x)=-2x\ell^{+,i}+\ell^{-,i}$. Let us denote by $s'_{n}(x)$ the spherical coefficients
of $R'$ at $x$. As special point we chose $x_{0}=\frac{1}{2}+\frac{\sqrt{3}}{2}i$.
By direct inspection, we have that
\begin{align*}
& &s'_{0}(i)=-1-\sqrt{3}i,\\
&s'_{1}(x)=(\de_{s}R')(x)=-1-\frac{\alpha i}{\beta}+\frac{i}{2\beta}, &s'_{1}(i)=-1+\frac{2\sqrt{3}}{3}i,\\
&s'_{2}(x)=(\de_{c}\de_{s}R')(x)=\left(1+\frac{I}{2\beta}\right)\frac{i}{2\beta}, &s'_{2}(i)=-\frac{1}{3}+\frac{\sqrt{3}}{3}i.\\
\end{align*}
Now, by Proposition~\ref{main} and Corollary~\ref{idempotents} we obtain
$$
s'_{2k}(x)=\frac{(2k-3)!!}{k!2^{k}\beta^{2k-1}}\left[\frac{(1+2\alpha)(2k-1)}{2\beta}I-1\right]i,
$$
and
$$
s'_{2k+1}=\frac{(1+2\alpha)(2k-1)!!}{2^{k+1}\beta^{2k+1}}i,
$$
and so
$$
s'_{2k}(x_{0})=-\frac{2^{k-1}\sqrt{3}(2k-3)!!}{k!3^{k}}\left[\frac{\sqrt{3}}{3}(2k-1)+i\right],\quad s'_{2k+1}=\frac{2^{k+1}\sqrt{3}(2k-1)!!}{k!3^{k+1}}i.
$$
Indeed we have that
$$
s_{1}(x_{0})+2\Im(x_{0})s_{2}(x_{0})=0+2\frac{\sqrt{3}i}{2}\left(-1+\frac{\sqrt{3}}{3}i\right)=-1-\sqrt{3}i=s'_{0}(x_{0}),
$$
$$
2s_{2}(x_{0})-2\Im(x_{0})s_{3}(x_{0})=2\left[-1-\frac{\sqrt{3}}{3}i-\frac{\sqrt{3}i}{2}\frac{\sqrt{3}i}{3}\right]=-\frac{3}{2}-\frac{2\sqrt{3}}{3}i=s'_{1}(x_{0}),
$$
$$
3s_{3}(x_{0})+4\Im(x_{0})s_{4}(x_{0})=3\frac{\sqrt{3}}{3}i+4\frac{\sqrt{3}i}{2}\left(\frac{1}{6}\left(-2+\frac{2\sqrt{3}}{3}i\right)\right)=-\frac{2}{3}+\frac{\sqrt{3}}{3}i=s'_{2}(x_{0}),
$$
and so on.
\end{example}

\begin{example}
For any $x=\alpha+I\beta\in\HH$ consider the quaternionic exponential function $\exp:\HH\to\HH$
$$
\exp(x)=e^{\alpha}(\cos(\beta)+I\sin(\beta)).
$$
Of course, for any $x\in\HH$, we have that $\de_{c}\exp(x)=\exp(x)$. Hence the formula in Theorem~\ref{coeffslder} can be used to compute the spherical
coefficients of the quaternionic exponential at any non-real point, i.e. for any integer $k$ we have
$$
\begin{cases}
s_{2k}=(2k+1)s_{2k+1}+2(k+1)\Im(q_{0})s_{2(k+1)},\\
s_{2k+1}=2(k+1)s_{2(k+1)}-2(k+1)\Im(q_{0})s_{2(k+1)+1},
\end{cases}
$$
or, equivalently
$$
\begin{cases}
s_{2k+2}=(2(k+1)\Im(q_{0}))^{-1}[s_{2k}-(2k+1)s_{2k+1}],\\
s_{2k+3}=(2(k+1)\Im(q_{0}))^{-1}[(2k+2)s_{2k+2}-s_{2k+1}].
\end{cases}
$$
\end{example}

It was proven by Perotti in~\cite{Perotti} that, for any slice regular function $f$ the functions $\de_{s}f$ and $(\de_{c}\de_{s})f$ are harmonic as
real functions $\R^{4}\to\R^{4}$. Unfortunately such property does not hold anymore for other iterations as already $\de_{s}\de_{c}\de_{s}f$ is not harmonic, in general
as the following example show.
\begin{example}
Let $x=\alpha+I\beta\in\HH$, the function $f(x)=x^{5}$ satisfies
\begin{align*}
(\de_{s}\de_{c}\de_{s})x^{5}&=(\de_{s}\de_{c})(5\alpha^{4}-10\alpha^{2}\beta^{2}+\beta^{4})=\de_{s}\frac{1}{2}\left(\frac{\de}{\de\alpha}-I\frac{\de}{\de\beta}\right)(5\alpha^{4}-10\alpha^{2}\beta^{2}+\beta^{4})\\
&=\de_{s}\frac{1}{2}(20\alpha^{3}-20\alpha\beta^{2}-I(-20\alpha^{2}\beta+4\beta^{3}))=10\alpha^{2}-2\beta^{2}.
\end{align*}
The result is a slice function and, for a slice function, the real Laplacian can be written as~\cite{Perotti}
$$
\Delta=4(\bar\de_{c}-\de_{s})(\de_{c}+\de_{s}).
$$
Hence, to compute $\Delta((\de_{s}\de_{c}\de_{s})x^{5})$, we perform the following calculation
\begin{align*}
\Delta((\de_{s}\de_{c}\de_{s})x^{5})&=4(\bar\de_{c}-\de_{s})(\de_{c}+\de_{s})(10\alpha^{2}-2\beta^{2})\\
&=4(\bar\de_{c}-\de_{s})\frac{1}{2}\left(\frac{\de}{\de\alpha}-I\frac{\de}{\de\beta}\right)(10\alpha^{2}-2\beta^{2})\\
&=4(\bar\de_{c}-\de_{s})(10\alpha+I2\beta)=4[5-1-2]=8.
\end{align*}
\end{example}

\appendix
\section{A flow chart of mutual implications of the results in Section~\ref{alternate}}\label{flow}
We present here a flow chart showing the mutual implications of the results needed to prove our first main theorem. 
We decided to present such a graph to justify the order and the way the results are presented.
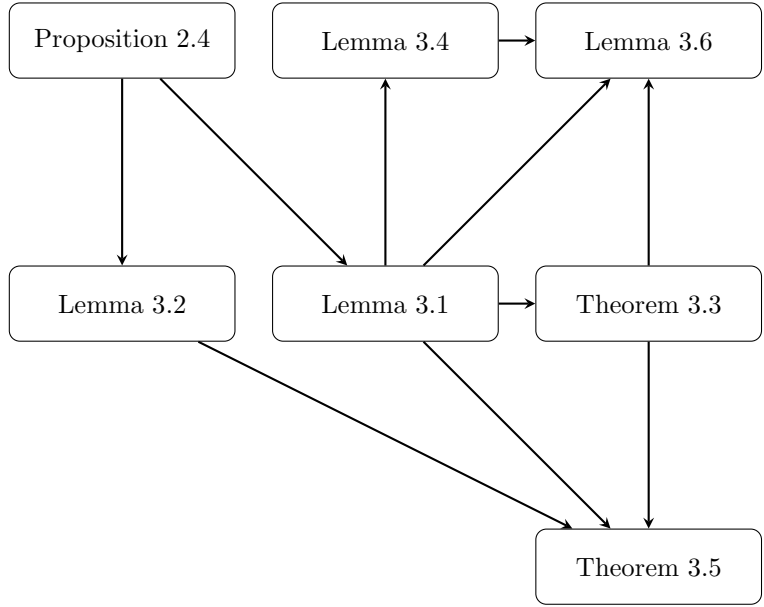
\begin{center}
\begin{tikzpicture}[node distance=3.5cm]
\node(prop)[startstop]{Proposition~\ref{main}};
\node(lemma5)[startstop, below of=prop]{Lemma~\ref{lemma5}};
\node(lemma2)[startstop, right of=lemma5]{Lemma~\ref{lemma2}};
\node(lemma3)[startstop,right of=prop]{Lemma~\ref{lemma3}};
\node(lemma8)[startstop, right of=lemma3]{Lemma~\ref{lemma8}};
\node(lemma6)[startstop, below of=lemma8]{Theorem~\ref{lemma6}};
\node(lemma7)[startstop, below of=lemma6]{Theorem~\ref{lemma7}};
\draw[arrow](prop)--(lemma2);
\draw[arrow](prop)--(lemma5);
\draw[arrow](lemma2)--(lemma6);
\draw[arrow](lemma2)--(lemma3);
\draw[arrow](lemma2)--(lemma8);
\draw[arrow](lemma2)--(lemma7);
\draw[arrow](lemma5)--(lemma7);
\draw[arrow](lemma6)--(lemma8);
\draw[arrow](lemma6)--(lemma7);
\draw[arrow](lemma3)--(lemma8);
\end{tikzpicture}
\end{center}

\section{Proof of Theorem~\ref{coeffslder} by explicit computations}\label{alternativeproof}

In this appendix we give an alternative proof of Theorem~\ref{coeffslder} in the same flavor of the computations given in Section~\ref{alternate}.
Such computations might be useful in the future to perform a much refined analysis of the previous results.

\begin{lemma}
Let $q_{0}\in\HH$ be any quaternion. For any integer $m\geq1$, we have the following equalities
$$
\de_{c}(\Delta_{q_{0}}^{m}(x))=2m\Delta_{q_{0}}^{m-1}(x)(x-\Re(q_{0})),\qquad \de_{c}(\Delta_{q_{0}}^{m}(x)(x-q_{0}))=\Delta_{q_{0}}^{m-1}(x)\left[m(x-q_{0})^{\cdot 2}+(m+1)\Delta_{q_{0}}(x)\right].
$$
\end{lemma}
\begin{proof}
Both equalities are direct consequences of the Leibniz rule for the slice derivative $\de_{c}$. In particular
$$
\de_{c}(\Delta_{q_{0}}^{m}(x))=m\Delta_{q_{0}}^{m-1}(x)\de_{c}\Delta_{q_{0}}(x)=m\Delta_{q_{0}}^{m-1}(x)[(x-q_{0})+(x-\bar q_{0})]=2m\Delta_{q_{0}}^{m-1}(x)(x-\Re(q_{0})),
$$
while
\begin{align*}
\de_{c}(\Delta_{q_{0}}^{m}(x)(x-q_{0}))&=\de_{c}(\Delta_{q_{0}}^{m}(x))(x-q_{0})+\Delta_{q_{0}}^{m}(x)=m\Delta_{q_{0}}^{m-1}(x)[(x-q_{0})+(x-\bar q_{0})]\cdot (x-q_{0})+\Delta_{q_{0}}^{m}(x)\\
&=\Delta_{q_{0}}^{m-1}(x)\left[m(x-q_{0})^{\cdot 2}+(m+1)\Delta_{q_{0}}(x)\right].
\end{align*}
\end{proof}
Hence, while the slice derivative of $\Delta_{q_{0}}^{m}(x)$ is the product of $\Delta_{q_{0}}^{m-1}(x)$ times a polynomial of degree $1$, the slice derivative
of $\Delta_{q_{0}}^{m}(x)(x-q_{0})$ can be written as the product of $\Delta_{q_{0}}^{m-1}(x)$ times the degree $2$ polynomial 
\begin{equation}\label{frakdef}
\mathfrak{P}_{q_{0},m}(x):=m(x-q_{0})^{\cdot 2}+(m+1)\Delta_{q_{0}}(x).
\end{equation}
For future convenience, we collect in the next Lemma some of the properties of $\mathfrak{P}_{q_{0},m}$.
\begin{lemma}\label{frakp}
Let $q_{0}=\alpha_{0}+I_{0}\beta_{0}\in\HH$ and $m\geq 1$ be an integer and $\mathfrak{P}_{q_{0},m}$ be the slice regular polynomial defined in Formula~\eqref{frakdef}. Then we have the following equalities
\begin{align}
(v_{s}\mathfrak{P}_{q_{0},m})(q_{0})&=2m(\Im(q_{0}))^{2}\label{eq1frak}\\
(\de_{s}\mathfrak{P}_{q_{0},m})(q_{0})&=-2m\Im(q_{0})\label{eq2frak}\\
(\de_{c}\mathfrak{P}_{q_{0},m})(q_{0})&=2(m+1)\Im(q_{0})\label{eq3frak}\\
(\de_{c}\de_{s})\mathfrak{P}_{q_{0},m}&=2m+1\label{eq4frak}\\
(\de_{s}\de_{c})\mathfrak{P}_{q_{0},m}&=2(2m+1)\label{eq5frak}
\end{align}
\end{lemma}
\begin{proof}
We prove all the equalities by direct inspection. We start from Formula~\eqref{eq1frak}. The spherical value of a slice function $g$ is defined as
$(v_{s}g)(x)=(g(x)+g(x^{c}))/2$, hence, since $(x-q_{0})^{\cdot 2}_{|_{x=q_{0}}}=0$ and $\Delta_{q_{0}}(q_{0})=\Delta_{q_{0}}(q_{0}^{c})=0$, we have
\begin{align*}
(v_{s}\mathfrak{P}_{q_{0},m})(q_{0})&=\frac{m((x-q_{0})^{\cdot 2})_{|_{x=q_{0}^{c}}}}{2}=\frac{m}{2}((q_{0}^{c})^{2}-2|q_{0}|^{2}+q_{0}^{2})\\
&=\frac{m}{2}2(\Re(q_{0}^{2})-|q_{0}|^{2})=m(-2\beta_{0}^{2})=2m(\Im(q_{0}))^{2}.
\end{align*}
We now pass to Formula~\eqref{eq2frak}. As the spherical derivative of $\Delta_{q_{0}}$ at $q_{0}$ is zero, we have
$$
(\de_{s}\mathfrak{P}_{q_{0},m})(q_{0})=(\de_{s}(m(x-q_{0})^{\cdot 2}))_{|_{x=q_0}} =m(2\Re(q_{0})-2q_{0})=-2m\Im(q_{0}).
$$

For Formula~\eqref{eq3frak} we have
\begin{align*}
(\de_{c}\mathfrak{P}_{q_{0},m})(q_{0})&=(\de_{c}(m(x-q_{0})^{\cdot 2}+(m+1)\Delta_{q_{0}}(x)))_{|_{x=q_{0}}}\\
&=(2m(x-q_{0})+(m+1)((x-q_{0})+(x-q_{0}^{c})))_{|_{x=q_{0}}}\\
&=(m+1)(q_{0}-q_{0}^{c})=2(m+1)\Im(q_{0}).
\end{align*}
Passing now to the mixed second derivatives, recalling Corollary~\ref{cor1} we have 
\begin{align*}
(\de_{c}\de_{s})\mathfrak{P}_{q_{0},m}&=m(\de_{c}\de_{s})(x-q_{0})^{\cdot 2}+(m+1)(\de_{c}\de_{s})\Delta_{q_{0}}\\
&=m\de_{c}(2\alpha-2q_{0})+(m+1)=m+m+1=2m+1,
\end{align*}
while
\begin{align*}
(\de_{s}\de_{c})\mathfrak{P}_{q_{0},m}&=m(\de_{s}\de_{c})(x-q_{0})^{\cdot 2}+(m+1)(\de_{s}\de_{c})\Delta_{q_{0}}\\
&=m\de_{s}(2x-2q_{0})+(m+1)2=2m+2m+2=2(2m+1),
\end{align*}
proving Formulas~\eqref{eq4frak} and~\eqref{eq4frak}.
\end{proof}

Now we are able to go forward in our principal quest. First of all we would need to compute $(\de_{c}\de_{s})^{k}(\de_{c}(\Delta_{q_{0}}^{m}))$ and $\de_{s}(\de_{c}\de_{s})^{k}(\de_{c}(\Delta_{q_{0}}^{m}))$, but these quantities, at least at $q_{0}$, were already computed in Lemmas~\ref{lemma5} and~\ref{lemma3}. We collect such information in the following remark.
\begin{remark}\label{remb3}
Let $q_{0}\in\HH$ and $m\in\mathbb{N}\setminus\{0\}$. Then, for any $0\leq k< m$ we have
$$
((\de_{c}\de_{s})^{k}(\de_{c}(\Delta_{q_{0}}^{m}(x))))_{|_{x=q_{0}}}=\begin{cases}
2\Im(q_{0})(m+1)!,\quad\mbox{ if } k=m-1,\\
0,\qquad\qquad\qquad\mbox{otherwise}.
\end{cases}
$$

Let $q_{0}\in\HH$ and $m\in\mathbb{N}\setminus\{0\}$. Then, for any $0\leq k< m$ we have
$$
(\de_{s}(\de_{c}\de_{s})^{k}(\de_{c}(\Delta_{q_{0}}^{m}(x))))_{|_{x=q_{0}}}=\begin{cases}
2m!,\quad\mbox{ if } k=m-1,\\
0,\qquad\qquad\mbox{otherwise}.
\end{cases}
$$
\end{remark}

We now pass to analyze $\de_{c}(\Delta_{q_{0}}^{m}(x)(x-q_{0}))$. Clearly, when $m=0$ we have
$\de_{c}(x-q_{0})=1$, while, for any $m\geq 1$, the equality
$$\de_{c}(\Delta_{q_{0}}^{m}(x)(x-q_{0}))=\Delta_{q_{0}}^{m-1}(x)\mathfrak{P}_{q_{0},m}(x),$$
implies that $(\de_{c}(\Delta_{q_{0}}^{m}(x)(x-q_{0})))_{|_{x=q_{0}}}=0$.
We also consider $(\de_{c}\de_{s})(\de_{c}(\Delta_{q_{0}}^{m}(x)(x-q_{0})))$.
Clearly, for $m=0$ the last vanishes identically. If $m=1$, then, thanks to Formula~\eqref{eq4frak},
\begin{equation}\label{eqbase}
(\de_{c}\de_{s})(\de_{c}(\Delta_{q_{0}}(x)(x-q_{0})))=(\de_{c}\de_{s})\mathfrak{P}_{q_{0},1}=3.
\end{equation}

Assume now that $m>1$. In this case, from Lemma~\ref{firstiteration}, we have 
\begin{align*}
(\de_{c}\de_{s})(\de_{c}(\Delta_{q_{0}}^{m}(x)(x-q_{0})))&=(\de_{c}\de_{s})(\Delta_{q_{0}}^{m-1}(x)\mathfrak{P}_{q_{0},m}(x))\\
&=((\de_{c}\de_{s})\Delta_{q_{0}}^{m-1}(x))v_{s}\mathfrak{P}_{q_{0},m}(x)+v_{s}\Delta_{q_{0}}^{m-1}(x)((\de_{c}\de_{s})\mathfrak{P}_{q_{0},m}(x))\\
&\quad\quad+\frac{1}{2}\left\{(\de_{s}\Delta_{q_{0}}^{m-1}(x))(\de_{c}\mathfrak{P}_{q_{0},m}(x))+(\de_{c}\Delta_{q_{0}}^{m-1}(x))(\de_{s}\mathfrak{P}_{q_{0},m})(x)\right\}.
\end{align*}
Therefore, when we evaluate at $x=q_{0}$, thanks to the fact that $(v_{s}\Delta_{q_{0}}^{m-1})(q_{0})=(\de_{s}\Delta_{q_{0}}^{m-1})(q_{0})=0$ and to Formulas~\eqref{eq1frak} and~\eqref{eq2frak}, we get
$$
((\de_{c}\de_{s})(\de_{c}(\Delta_{q_{0}}^{m}(x)(x-q_{0}))))_{|_{x=q_{0}}}=(((\de_{c}\de_{s})\Delta_{q_{0}}^{m-1}(x)))_{|_{x=q_{0}}}2m(\Im(q_{0}))^{2}+\frac{1}{2}((\de_{c}\Delta_{q_{0}}^{m-1}(x))_{|_{x=q_{0}}}(-2m\Im(q_{0}))).
$$
Now, if $m-1>1$, then everything vanishes, while if $m=2$, thanks to Theorem~\ref{lemma6} and to Remark~\ref{remb3}, we get
$$
((\de_{c}\de_{s})(\de_{c}(\Delta_{q_{0}}^{2}(x)(x-q_{0}))))_{|_{x=q_{0}}}=2m(\Im(q_{0}))^{2}+\frac{1}{2}(2\Im(q_{0})(-2m\Im(q_{0})))=0.
$$
We can now state the result in general.
\begin{proposition}\label{proposition3}
Let $q_{0}\in\HH$ and $m\in\mathbb{N}\setminus\{0\}$. Then, for any $0\leq k\leq m$ we have
$$
((\de_{c}\de_{s})^{k}(\de_{c}(\Delta_{q_{0}}^{m}(x-q_{0}))))_{|_{x=q_{0}}}=\begin{cases}
(2m+1)m!,\quad\mbox{ if } k=m,\\
0,\qquad\qquad\qquad\mbox{otherwise}.
\end{cases}
$$
\end{proposition}
\begin{proof}
We already analyzed the cases when $k=0,1$ before. Assume then that $k>1$.
From Proposition~\ref{main}, we derive
\begin{align*}
(\de_{c}\de_{s})^{k}(\de_{c}(\Delta_{q_{0}}^{m}(x)(x-q_{0})))&=(\de_{c}\de_{s})^{k}(\Delta_{q_{0}}^{m-1}(x)\mathfrak{P}_{q_{0},m}(x))\\
&=((\de_{c}\de_{s})^{k}\Delta_{q_{0}}^{m-1}(x))v_{s}\mathfrak{P}_{q_{0},m}(x)\\
&\quad+\frac{1}{2}\left\{(\de_{s}(\de_{c}\de_{s})^{k-1}\Delta_{q_{0}}^{m-1}(x))\de_{c}\mathfrak{P}_{q_{0},m}(x)
+(\de_{c}(\de_{s}\de_{c})^{k-1}\Delta_{q_{0}}^{m-1}(x))\de_{s}\mathfrak{P}_{q_{0},m}(x)\right.\\
&\quad\left.+(k-1)((\de_{c}\de_{s})^{k-1}\Delta_{q_{0}}^{m-1}(x))((\de_{s}\de_{c})\mathfrak{P}_{q_{0},m}(x))+((\de_{s}\de_{c})^{k-1}\Delta_{q_{0}}^{m-1}(x))((\de_{c}\de_{s})\mathfrak{P}_{q_{0},m}(x))
\right\}.
\end{align*}
Now, for $x=q_{0}$, taking into account the formulas in Lemma~\ref{frakp}, we obtain
\begin{align*}
(\de_{c}\de_{s})^{k}(\Delta_{q_{0}}^{m}(x)(x-q_{0}))_{|_{x=q_{0}}}&=((\de_{c}\de_{s})^{k}\Delta_{q_{0}}^{m-1}(x))_{|_{x=q_{0}}}2m(\Im(q_{0}))^{2}\\
&\quad+\frac{1}{2}\left\{(\de_{s}(\de_{c}\de_{s})^{k-1}\Delta_{q_{0}}^{m-1}(x))_{|_{x=q_{0}}}2(m+1)\Im(q_{0})+\right.\\
&\quad+(\de_{c}(\de_{s}\de_{c})^{k-1}\Delta_{q_{0}}^{m-1}(x))_{|_{x=q_{0}}}(-2m\Im(q_{0}))\\
&\quad\left.+(k-1)((\de_{c}\de_{s})^{k-1}\Delta_{q_{0}}^{m-1}(x))_{|_{x=q_{0}}}2(2m+1)+((\de_{s}\de_{c})^{k-1}\Delta_{q_{0}}^{m-1}(x))_{|_{x=q_{0}}}(2m+1)
\right\}.
\end{align*}
Clearly, when $k<m-1$ everything vanishes. When $k=m-1$, using Theorem~\ref{lemma6} and Lemmas~\ref{lemma2}, ~\ref{lemma3}, ~\ref{lemma5}, we are left with
\begin{align*}
(\de_{c}\de_{s})^{m-1}(\Delta_{q_{0}}^{m}(x)(x-q_{0}))_{|_{x=q_{0}}}&=((\de_{c}\de_{s})^{m-1}\Delta_{q_{0}}^{m-1}(x))_{|_{x=q_{0}}}2m(\Im(q_{0}))^{2}+\\
&\quad\frac{1}{2}\left\{(\de_{c}(\de_{s}\de_{c})^{m-2}\Delta_{q_{0}}^{m-1}(x))_{|_{x=q_{0}}}(-2m\Im(q_{0}))\right\}\\
&=(m-1)!2m(\Im(q_{0}))^{2}+\frac{1}{2}\left\{2\Im(q_{0})(m-1)!(-2m\Im(q_{0}))\right\}=0.
\end{align*}
Finally, when $k=m$, we are left with
$$
(\de_{c}\de_{s})^{m}(\Delta_{q_{0}}^{m}(x)(x-q_{0}))_{|_{x=q_{0}}}=\frac{1}{2}\left\{(m-1)((\de_{c}\de_{s})^{m-1}\Delta_{q_{0}}^{m-1}(x))_{|_{x=q_{0}}}2(2m+1)+((\de_{s}\de_{c})^{m-1}\Delta_{q_{0}}^{m-1}(x))_{|_{x=q_{0}}}(2m+1)
\right\}.
$$
Therefore, by Lemma~\ref{lemma5} and Theorem~\ref{lemma6}, we obtain
\begin{align*}
(\de_{c}\de_{s})^{m}(\Delta_{q_{0}}^{m}(x)(x-q_{0}))_{|_{x=q_{0}}}&=\frac{1}{2}\left\{(m-1)(m-1)!2(2m+1)+2(m-1)!(2m+1)
\right\}\\
&=(m-1)(m-1)!(2m+1)+(m-1)!(2m+1)\\
&=(2m+1)(m-1)!(m-1+1)=(2m+1)m!
\end{align*}
\end{proof}
We are left with the following cases.
\begin{proposition}\label{proposition4}
Let $q_{0}\in\HH$ and $m\in\mathbb{N}\setminus\{0\}$. Then, for any $0\leq k< m$ we have
$$
(\de_{s}(\de_{c}\de_{s})^{k}(\de_{c}(\Delta_{q_{0}}^{m}(x)(x-q_{0}))))_{|_{x=q_{0}}}=\begin{cases}
-2\Im(q_{0})m!,\quad\mbox{ if } k=m-1,\\
0,\qquad\qquad\qquad\mbox{otherwise}.
\end{cases}
$$
\end{proposition}
\begin{proof}
Firstly we address the case $k=0$. In this case we have
$$
(\de_{s}(\de_{c}(\Delta_{q_{0}}^{m}(x)(x-q_{0}))))=\de_{s}(\Delta_{q_{0}}^{m-1}(x)\mathfrak{P}_{q_{0},m}(x)),
$$
and as , for any $p>0$, we have that $(v_{s}\Delta_{q_{0}}^{p})(q_{0})=(\de_{s}\Delta_{q_{0}}^{p})(q_{0})=0$, the only case to inspect is when
$m-1=0$, i.e., $m=1$. In this case, thanks to Formula~\eqref{eq2frak}, we have
$$
(\de_{s}(\de_{c}(\Delta_{q_{0}}(x)(x-q_{0}))))_{|_{x=q_{0}}}=\de_{s}(\mathfrak{P}_{q_{0},1})(q_{0})=-2\Im(q_{0}).
$$
Assume now that $k=1$ and $m>1$. In this case, from Lemma~\ref{firstiteration}, we have 
\begin{align*}
\de_{s}(\de_{c}\de_{s})(\de_{c}(\Delta_{q_{0}}^{m}(x)(x-q_{0})))&=\de_{s}(\de_{c}\de_{s})(\Delta_{q_{0}}^{m-1}(x)\mathfrak{P}_{q_{0},m}(x))\\
&=(\de_{s}(\de_{c}\de_{s})\Delta_{q_{0}}^{m-1}(x))v_{s}\mathfrak{P}_{q_{0},m}(x)+\\
&\quad +\frac{1}{2}\left\{(\de_{s}\Delta_{q_{0}}^{m-1}(x))(\de_{s}\de_{c}\mathfrak{P}_{q_{0},m}(x))+(\de_{s}\de_{c}\Delta_{q_{0}}^{m-1}(x))(\de_{s}\mathfrak{P}_{q_{0},m}(x))\right\}.
\end{align*}
If $m=2$, then, thanks to Lemma~\ref{lemma2} and the fact that $(\de_{s}\Delta_{q_{0}})(q_{0})=0$, at $q_{0}$, we have 
$$
(\de_{s}(\de_{c}\de_{s})(\de_{c}(\Delta_{q_{0}}^{2}(x)(x-q_{0}))))_{|_{x=q_{0}}}=\frac{1}{2}\left\{(\de_{s}\de_{c}\Delta_{q_{0}}(x))(\de_{s}\mathfrak{P}_{q_{0},2}(x))\right\}_{|_{x=q_{0}}}=\frac{1}{2}\left\{2(-4\Im(q_{0}))\right\}=-4\Im(q_{0}),
$$
where the last equalities are due to Corollary~\ref{cor1} and Formula~\eqref{eq2frak}.

Assume now that $k>0$ and $m>k$. In this case, thanks to Proposition~\ref{main}, we have
\begin{align*}
\de_{s}(\de_{c}\de_{s})^{k}(\de_{c}(\Delta_{q_{0}}^{m}(x)(x-q_{0})))&=\de_{s}(\de_{c}\de_{s})^{k}(\Delta_{q_{0}}^{m-1}(x)\mathfrak{P}_{q_{0},m}(x))\\
&=(\de_{s}(\de_{c}\de_{s})^{k}\Delta_{q_{0}}^{m-1}(x))v_{s}\mathfrak{P}_{q_{0},m}(x)\\
&\quad+\frac{1}{2}\left\{(\de_{s}(\de_{c}\de_{s})^{k-1}\Delta_{q_{0}}^{m-1}(x))(\de_{s}\de_{c})\mathfrak{P}_{q_{0},m}(x)
+((\de_{s}\de_{c})^{k}\Delta_{q_{0}}^{m-1}(x))\de_{s}\mathfrak{P}_{q_{0},m}(x)\right.\\
&\quad\left.+(k-1)(\de_{s}(\de_{c}\de_{s})^{k-1}\Delta_{q_{0}}^{m-1}(x))((\de_{s}\de_{c})\mathfrak{P}_{q_{0},m}(x))
\right\}\\
&=(\de_{s}(\de_{c}\de_{s})^{k}\Delta_{q_{0}}^{m-1}(x))v_{s}\mathfrak{P}_{q_{0},m}(x)\\
&\quad+\frac{1}{2}\left\{k(\de_{s}(\de_{c}\de_{s})^{k-1}\Delta_{q_{0}}^{m-1}(x))(\de_{s}\de_{c})\mathfrak{P}_{q_{0},m}(x)
+((\de_{s}\de_{c})^{k}\Delta_{q_{0}}^{m-1}(x))\de_{s}\mathfrak{P}_{q_{0},m}(x)\right\}.
\end{align*}
Now, for any combination of $h$ and $m$, thanks to Lemma~\ref{lemma2}, $(\de_{s}(\de_{c}\de_{s})^{h}\Delta_{q_{0}}^{m-1})$ vanishes at $q_{0}$. This is true, in particular, if $h=k, k-1$.
Moreover, thanks to Lemma~\ref{lemma5}, we have that $((\de_{s}\de_{c})^{k}\Delta_{q_{0}}^{m-1})(q_{0})$ vanishes whenever $k\neq m-1$.
Therefore, the only non-zero case is $k=m-1$, we get $2(m-1)!$. In such a case, thanks to Formula~\eqref{eq2frak}, we have
\begin{align*}
(\de_{s}(\de_{c}\de_{s})^{m-1}(\de_{c}(\Delta_{q_{0}}^{2}(x)(x-q_{0}))))_{|_{x=q_{0}}}&=\frac{1}{2}\left\{((\de_{s}\de_{c})^{m-1}\Delta_{q_{0}}^{m-1}(x))\de_{s}\mathfrak{P}_{q_{0},m}(q_{0})\right\}\\
&=\frac{1}{2}\left\{2(m-1)!(-m2\Im(q_{0}))\right\}=-2m!\Im(q_{0}).
\end{align*}
\end{proof}

As before, for the convenience of the reader, we collect here all the previous computations in a handy table.
All the cases that do not appear are the vanishing ones.
\begin{table}[h!]
\centering
\begin{tabular}{ |c||c|c|  }
 \hline
 \multicolumn{2}{|c|}{A summary of the previous results} \\
 \hline
 Result &Reference\\
 \hline
$((\de_{c}\de_{s})^{k}(\de_{c}(\Delta_{q_{0}}^{k+1}(x))))_{|_{x=q_{0}}}=2\Im(q_{0})(k+1)!$ & Lemma~\ref{lemma3}\\
$((\de_{c}\de_{s})^{k}(\de_{c}(\Delta_{q_{0}}^{k}(x)(x-q_{0}))))_{|_{x=q_{0}}}=(2k+1)k!$ & Proposition~\ref{proposition3}\\
$(\de_{s}(\de_{c}\de_{s})^{k}(\de_{c}(\Delta_{q_{0}}^{k+1}(x))))_{|_{x=q_{0}}}=2(k+1)!$ & Lemma~\ref{lemma5}\\
$(\de_{s}(\de_{c}\de_{s})^{k}(\de_{c}(\Delta_{q_{0}}^{k+1}(x)(x-q_{0}))))_{|_{x=q_{0}}}=-2\Im(q_{0})(k+1)!$ & Proposition~\ref{proposition4}\\
 \hline
\end{tabular}
\end{table}

\begin{proof}[Alternative proof of Theorem~\ref{coeffslder}]
The proof is quite straightforward after the previous results in this appendix.
In particular by checking the previous table, we have that 
\begin{align*}
s_{2k}'=\frac{1}{k!}((\de_{c}\de_{s})^{k}\de_{c}f)(q_{0})&=\frac{1}{k!}\left[((\de_{c}\de_{s})^{k}(\de_{c}(\Delta_{q_{0}}^{k}(x)(x-q_{0}))))_{|_{x=q_{0}}}s_{2k+1}+((\de_{c}\de_{s})^{k}(\de_{c}(\Delta_{q_{0}}^{k+1}(x))))_{|_{x=q_{0}}}s_{2(k+1)}\right]\\
&=\frac{1}{k!}\left[(2k+1)k!s_{2k+1}+2\Im(q_{0})(k+1)!s_{2(k+1)}\right]\\
&=(2k+1)s_{2k+1}+2(k+1)\Im(q_{0})s_{2(k+1)},
\end{align*}
while
\begin{align*}
s_{2k+1}'=\frac{1}{k!}(\de_{s}(\de_{c}\de_{s})^{k}\de_{c}f)(q_{0})&=\frac{1}{k!}\left[(\de_{s}(\de_{c}\de_{s})^{k}(\de_{c}(\Delta_{q_{0}}^{k+1}(x))))_{|_{x=q_{0}}}s_{2(k+1)}\right.\\
&\qquad\left.+(\de_{s}(\de_{c}\de_{s})^{k}(\de_{c}(\Delta_{q_{0}}^{k+1}(x)(x-q_{0}))))_{|_{x=q_{0}}}s_{2(k+1)+1}\right]\\
&=\frac{1}{k!}\left[2(k+1)!s_{2(k+1)}-2\Im(q_{0})(k+1)!s_{2(k+1)+1}\right]\\
&=2(k+1)s_{2(k+1)}-(k+1)2\Im(q_{0})s_{2(k+1)+1}.,
\end{align*}
concluding the proof.
\end{proof}

\bibliographystyle{amsplain}

\begin{thebibliography}{Ci-Kr}



\bibitem{A-CVEE}
A. Altavilla, Some properties for quaternionic slice-regular functions on domains without real points. 
Complex Var. Elliptic Equ. 60, n. 1 (2015), 59--77.

\bibitem{altavilladiff}
A. Altavilla.
On the real differential of a slice regular function.
Adv. Geom. 18 (2018), no. 1, 5--26. 


\bibitem{A-dF} A. Altavilla, C. de Fabritiis, $*$-exponential of slice-regular functions, Proc. Amer. Math. Soc. 147, 2019, 1173--1188.
\bibitem{A-dFAMPA} A. Altavilla, C. de Fabritiis, s-Regular functions which preserve a complex slice, Ann. Mat. Pura Appl. (4) 197:4, 2018, 1269--1294.
\bibitem{A-dFLAA} A. Altavilla, C. de Fabritiis,
Equivalence of slice semi-regular functions via Sylvester operators
Linear Algebra and Its Applications, 2020, 607, pp. 151--189.
\bibitem{A-dFConc} A. Altavilla, C. de Fabritiis, Applications of the Sylvester operator in the space of slice semi-regular functions
Concrete Operators, 2020, 7(1), pp. 1--12.







\bibitem{A-S}A. Altavilla and G. Sarfatti. Slice-Polynomial Functions and Twistor Geometry of Ruled Surfaces in $\mathbb{CP}^{3}$. Math. Z. 291(3-4) (2019), 1059--1092.





\bibitem{douren1}
X. Dou and G. Ren. Riemann slice-domains over quaternions I.
arXiv:1808.06994 [math.CV], 2018.
\bibitem{douren2}
X. Dou and G. Ren. Riemann slice-domains over quaternions II.
arXiv:1808.06994 [math.CV], 2018.
\bibitem{dourensabadini}
X. Dou, G. Ren, and I. Sabadini. Extension theorem and representation formula in non-axially symmetric domains for slice regular functions.
arXiv:2003.10487 [math.CV], 2020, to appear in Journal of the European Mathematical Society.
\bibitem{G-St} G.Gentili, D. C. Struppa, A new theory of regular functions of a quaternionic variable,  Adv. Math. 216 (2007), no. 1, 279--301.
\bibitem{G-SJMAA} G. Gentili, C. Stoppato, Geometric function theory over quaternionic slice domains,
J. Math. Anal. Appl. 495 (2021), no. 2, 124780, https://doi.org/10.1016/j.jmaa.2020.124780.
\bibitem{gentilistoppatoPAMS}
G. Gentili and C. Stoppato. A local representation formula for quaternionic
slice regular functions. Proc. Amer. Math. Soc. 149 (2021), no. 5, 2025--2034.

\bibitem{G-S-St} G. Gentili, C. Stoppato, D. C. Struppa, Regular Functions of a Quaternionic Variable,  Springer Monographs in Mathematics, Springer, 2013. 
\bibitem{ghilonipoly} R. Ghiloni, Slice-by-slice and global smoothness of slice regular and polyanalytic functions, arXiv:2011.09919 [math.CV], 2020.



\bibitem{G-P} R. Ghiloni, A. Perotti, Slice regular functions on real alternative algebras, Adv. in Math., v. 226, n. 2 (2011),  1662-1691.

\bibitem{ghiloniperottiseveral1}
R. Ghiloni, A. Perotti, Slice regularity in several variables. In Progress in analysis.
Proceedings of the 8th congress of the International Society for Analysis, its Applications,
and Computation (ISAAC), Moscow, Russia, August 22--27, 2011. Volume 1, pages 179--
186. Moscow: Peoples Friendship University of Russia, 2012

\bibitem{G-P-series} R. Ghiloni, A. Perotti, Power and spherical series over real alternative *-algebras, Indiana University Mathematics Journal, Vol. 63, No. 2, Pages 495-532 (2014)
\bibitem{ghiloniperottiseveral2} R. Ghiloni, A. Perotti, Slice regular functions in several variables (2020), available at arXiv:2007.14925.

\bibitem{GPSalgebra} R. Ghiloni, A. Perotti, C. Stoppato, The algebra of slice functions, Trans. of Amer. Math. Soc., Volume 369, N.7, 2017, pp.4725--4762.
\bibitem{GPSadvances} R. Ghiloni, A. Perotti, and C. Stoppato. Singularities of slice regular functions over real
alternative $*$-algebras. Adv. Math., 305:1085--1130, 2017.

\bibitem{GPSdivision} 
R. Ghiloni., A. Perotti, C. Stoppato, Division algebras of slice functions. Proceedings of the Royal Society of Edinburgh: Section A Mathematics, 150(4), (2020), 2055--2082. doi:10.1017/prm.2019.13.

\bibitem{mongodi} Mongodi, S. Holomorphicity of Slice-Regular Functions. Complex Anal. Oper. Theory 14, 37 (2020). https://doi.org/10.1007/s11785-020-00996-2

\bibitem{M-S}
A.~Monguzzi, G.~Sarfatti, Shift invariant subspaces of slice $L^{2}$ functions. Ann. Acad. Sci. Fenn. Math. 43 (2018), 1045--1061.

\bibitem{Perotti} A. Perotti,  Slice Regularity and Harmonicity on Clifford Algebras. In: Bernstein S. (eds) Topics in Clifford Analysis. Trends in Mathematics. Birkh\"auser, Cham. 
https://doi.org/10.1007/978-3-030-23854-4\_3



\bibitem{stoppatosph} C. Stoppato, A new series expansion for slice regular functions. 
Adv. Math. 231 (2012), no. 3-4, 1401--1416.




\end{thebibliography}

\end{document}